\documentclass[10pt]{amsart}

\usepackage{amssymb,verbatim,url}
\usepackage[pdftex]{graphicx}
\usepackage{blkarray}
\usepackage[all]{xy}

\newcommand{\nc}{\newcommand}
\nc{\fix}[1]{\mathcal{F}(#1)}
\nc{\fixx}{\mathcal{F}}
\nc{\cG}{\mathcal{G}}
\nc{\C}{\mathbf{C}}
\nc{\Q}{\mathbf{Q}}
\nc{\Z}{\mathbf{Z}}
\nc{\F}{\mathbf{F}}
\nc{\R}{\mathbf{R}}
\nc{\Qp}{\mathbf{Q}_p}
\nc{\gen}[1]{\langle #1\rangle}
\nc{\ds}{\displaystyle}
\nc{\Oc}{\mathcal{O}}
\nc{\OF}{\Oc_F}
\nc{\OL}{\Oc_L}
\nc{\cL}{\mathcal{L}}
\nc{\cK}{\mathcal{K}}
\nc{\OLg}{\Oc_{L^g}}
\nc{\fP}{\mathfrak{P}}
\nc{\fp}{\mathfrak{p}}
\nc{\fd}{\mathfrak{D}}
\nc{\fr}{\mathfrak{d}}
\nc{\ft}{\mathfrak{t}}
\nc{\fc}{\mathfrak{c}}
\nc{\g}{{\rm gal}}
\nc{\ta}{{\,t}}
\nc{\GL}{\textrm{GL}}
\nc{\PGL}{\textrm{PGL}}
\nc{\PSL}{\textrm{PSL}}
\nc{\SL}{\textrm{SL}}
\nc{\bbar}[1]{\overline{#1}}
\nc{\vd}{\vec{d}\, }
\nc{\into}{\hookrightarrow}
\nc{\cots}{, \ldots, }
\nc{\pots}{+ \cdots +}
\nc{\seq}{\subseteq}
\nc{\note}[1]{\begin{center}\fbox{\textbf{#1}}\end{center}}
\nc{\cmmt}[1]{}
\nc{\bp}{\bullet \!\!\!\! +} 
\nc{\sst}{\scriptstyle}
\nc{\emm}{\em}
\thanks{This work was partially supported by a grant from the Simons Foundation (\#209472 to David Roberts).}

\DeclareMathOperator{\gal}{Gal}

\DeclareMathOperator{\Hom}{Hom}

\newtheorem{thm}{Theorem}[section]
\newtheorem{prop}[thm]{Proposition}
\newtheorem{lemma}[thm]{Lemma}
\newtheorem{cor}[thm]{Corollary}
\newtheorem{defn}[thm]{Definition}

\numberwithin{table}{section}
\numberwithin{figure}{section}
\numberwithin{equation}{section}

\theoremstyle{remark}

\title[The Tame-Wild Principle]{The Tame-Wild Principle for
  Discriminant Relations for Number Fields} \author{John W.\ Jones}
\address{School of Mathematical and Statistical Sciences, Arizona
  State University, PO Box 871804, Tempe, AZ 85287} \email{jj@asu.edu}

\author{David P.\ Roberts}
\address{Division of Science and Mathematics, University of
  Minnesota-Morris, Morris, MN 56267}
\email{roberts@morris.umn.edu}

\begin{document}
\begin{abstract}
     Consider tuples $(K_1, \dots,K_r)$ of separable algebras
 over a common local or global number field $F$, with the $K_i$
 related to each other by specified resolvent constructions.   
 Under the assumption that all ramification is tame,
 simple group-theoretic calculations give best
 possible divisibility relations among the discriminants
 of $K_i/F$.  We show that for many resolvent constructions, 
 these divisibility
 relations continue to hold even in the presence of
 wild ramification.   
\end{abstract}

\maketitle

\section{Overview}
\label{overview}
     Let $G$ be a finite group and let $\phi_1$,\dots,$\phi_r$ be 
 permutation characters of $G$.  We say that
 a tuple $(K_1,\dots,K_r)$ of separable algebras over
 a common ground field $F$  has type 
 $(G, \phi_1,\dots,\phi_r)$ if for a joint splitting field $K^{\rm gal}$
 one can identify $\gal(K^{\rm gal}/F)$ with a subgroup of $G$
 such that the action of $\gal(K^{\rm gal}/F)$ on $\Hom_F(K_i,K^{\rm gal})$
 has character $\phi_i$.

     When $F$ is a local or global number field, one has discriminants 
  $\fd_{K_i/F}$ which are ideals in the ring of integers $\OF$ of $F$.
  One can ask for the strongest divisibility relations among these
  discriminants which hold as $(K_1,\dots,K_r)/F$ varies 
  over all possibilities of
  a given type $(G, \phi_1,\dots,\phi_r)$.  
  This question has a simple group-theoretic answer if one
  restricts attention to tuples for  which all ramification in each
  $K_i/F$ is tame.  
    
     This paper focuses on the following phenomenon:  {\em for many
 $(G, \phi_1,\dots,\phi_r)$, the divisibility relations for 
 tame $(K_1,\dots,K_r)/F$ of type $(G, \phi_1,\dots,\phi_r)$ hold also for 
 arbitrary $(K_1,\dots,K_r)/F$ of type $(G, \phi_1,\dots, \phi_r)$.}  
 In this case, we say that the tame-wild principle holds for 
 $(G, \phi_1,\dots,\phi_r)$.  Our terminology ``tame-wild
principle'' is intended to be reminiscent of the standard terminology 
``local-global principle'':   we are showing in this paper that simple 
tame computations can often solve a complicated wild problem,
just as simple local computations can often
solve a complicated global problem.  

     Section~\ref{introductory} provides an  
     introductory example. Section~\ref{chartheory}
     reviews some ramification theory centering
     on Artin characters, placing it in a 
     framework which will be convenient for us.  
     Section~\ref{tamewild} states the tame-wild principle and
     gives two simple methods for proving instances 
     of it.  
     
  If the tame-wild principle  holds for a fixed
  $G$ and any $(\phi_1,\dots,\phi_r)$ then we say it holds 
  universally for $G$.      Section~\ref{universalC} proves that the tame-wild principle holds universally
   for $G$ in a small class of groups we call U-groups.
   Section~\ref{universalN} considers the remaining groups, called
   N-groups, finding that the tame-wild principle
   usually does not hold universally for them.  
   
   Sections~\ref{comparing} and \ref{general} return to the more practical
   situation where one is given not only $G$ but
   also a small list of naturally arising $\phi_i$.  
   Our theme is that the tame-wild principle
   is likely to hold, despite the negative 
   results on N-groups.   Section~\ref{comparing}
   focuses on comparing an arbitrary
   algebra $K/F$ with its splitting field $K^{\rm gal}/F$,
   proving that one of the two divisibility statements coming from the tame-wild principle 
   holds for arbitrary $G$.  Section~\ref{general} gives a
   collection of examples exploring the range of $(G,\phi_1,\dots,\phi_r)$ 
   for which the tame-wild principle holds.  
        
     This paper is written with applications to tabulating 
  number fields of small discriminant in mind.  The
   topics in  \S\ref{intronf}, \S\ref{left}, and 
   \S\ref{aff3} all relate to this application.   Moreover,
   as we will make clear, the theory we present
   here still applies when
permutation characters $\phi_i$ are replaced by
  general characters $\chi_i$, and
     discriminants are correspondingly generalized to 
     conductors.   Applications to
     Artin $L$-functions of small conductor will
     be presented elsewhere.  
   
     We thank Nigel Boston and Griff Elder for
  helpful conversations.

\section{An introductory example}
\label{introductory}  
In \S\ref{introtamewild}, we provide an introductory instance of the tame-wild principle
 that we will revisit later to provide simple illustrations of 
 general points.     In \S\ref{intronf}, we illustrate how this instance
of the tame-wild principle gives an indirect but efficient way of 
solving a standard problem in   
tabulating number fields. 

\subsection{The tame-wild principle for $(S_5,\phi_5,\phi_6)$}
\label{introtamewild}   

\subsubsection*{The Cayley-Weber type}
 For our introductory example, we take the type $(S_5,\phi_5,\phi_6)$, 
 where $\phi_5$ is the character of the given degree five
 permutation representation, and $\phi_6$ is the character of the degree 
 six representation $S_5 \stackrel{\sim}{\rightarrow} PGL_2(5) \subset S_6$.  
 A pair of algebras $(K_5,K_6)$ has type $(S_5, \phi_5,\phi_6)$
  exactly when $K_6$ is the Cayley-Weber resolvent, as in e.g.\ \cite[\S2.3]{generic}, of $K_5$. 
  An explicit example 
  over $\Q$ is given by $K_n = \Q[x]/f_n(x)$ with
  \begin{align}
\label{f51} f_5(x) & =  x^5-2 x^4+4 x^3-4 x^2+2 x-4,  & D_5 & = 2^8 3^4 5^1, \\
\label{f61} f_6(x) & =  x^6-2 x^5+4 x^4-4 x^3+2 x^2-4 x-6, & D_6 & = 2^{10} 3^4 5^3.
 \end{align} 
 In this example, the Galois group is all of $S_5$, discriminants $\fd_{K_n/\Q} = (D_n)$ 
 are as indicated, and
 ramification is wild at $2$ and tame at $3$ and $5$.   We are concerned
 with exponent pairs $(a_\fp,b_\fp)$ on discriminants.   Here $(a_2,b_2)=(8,10)$, $(a_3,b_3) = (4,4)$,
 $(a_5,b_5) = (1,3)$, and otherwise $(a_\fp,b_\fp)=(0,0)$.  
  
  \subsubsection*{All possibilities for $(a_\fp,b_\fp)$}  Figure~\ref{fivetosix} 
  gives all nonzero possibilities for $(a_\fp,b_\fp)$ over $\Q$.   The fact that the 
  tame list is complete is immediate from the general formalism of the next section.  
  A brute force proof that the wild list is complete goes as follows:    
  \begin{figure}[htb]
\includegraphics[width=4.5in]{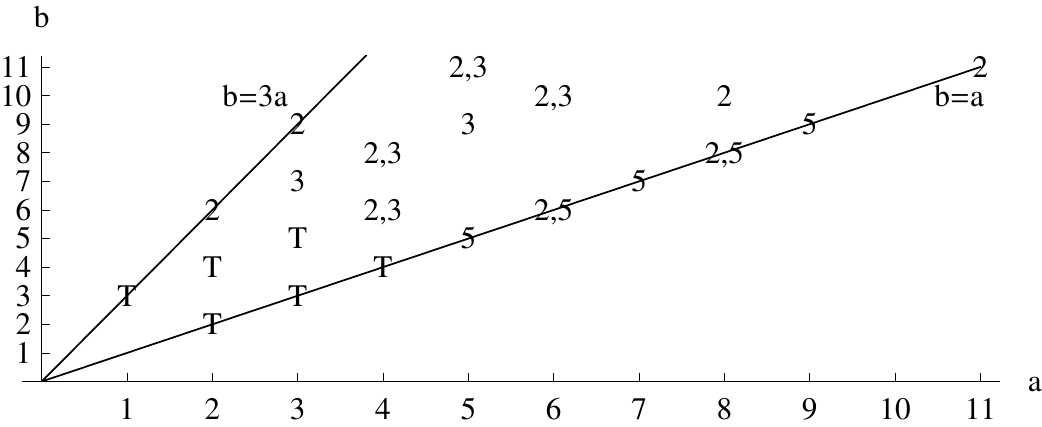}
\caption{\label{fivetosix}  The complete list of pairs $(a,b) \neq
  (0,0)$ which occur as $(a_\fp,b_\fp)$ for $(S_5,\phi_5,\phi_6)$ over
  $\Q$.      The pairs labeled T can occur with tame ramification, 
while the others can only occur for wild $p$-adic ramification as indicated.  }
\end{figure}
there are $113$, $57$, and $51$ quintic algebras $K_5$ over $\Q_p$
for $p=2$, $3$, and $5$ respectively \cite{jr-local-database};  for each, one can compute
$K_6$ and thus the pairs $(a_\fp,b_\fp)$;  the lists arising
are the ones drawn in Figure~\ref{fivetosix}.
For larger number fields $F$, the list of possibilities for 
tame $(a_\fp,b_\fp)$ is exactly the same, but the list of 
possibilities for wild $(a_\fp,b_\fp)$ grows without bound.

\subsubsection*{The tame-wild principle}
Figure~\ref{fivetosix} and the comment about general base fields $F$ 
clearly illustrate two general phenomena
about exponent vectors $(a_{\fp},b_{\fp})$.   First, in absolute terms, 
the exponents can be much larger in wild cases 
than they are in all tame cases.  But second, in relative terms, one can hope that 
the ratios
$a_{\fp}/b_{\fp}$ are quite similar in the wild and tame cases.  
We are interested in this paper only in the second phenomenon
and so we systematically consider ratios. 

In our example, the tame-wild principle is the statement
that 
\begin{equation}
\label{specificbounds}
\frac{1}{3} b_\fp \leq a_\fp \leq b_\fp
\end{equation}
holds for all $(K_5,K_6)/F$ of type $(S_5,\phi_5,\phi_6)$ 
and all primes $\fp$ of $F$.   In other words, when $b_\fp \neq 0$ one
must have $a_\fp/b_\fp \in [1/3,1]$.  
We have summarized a proof that \eqref{specificbounds} holds when one restricts
$F$ to be $\Q$ or one of the $\Q_p$.   We will see by a group-theoretic argument in 
 \S\ref{calculations}, not involving inspecting wild ramification at all, 
that \eqref{specificbounds} holds for general $F$.    However the situation is
subtle, as the analog of \eqref{specificbounds}  
holds for many $(G,\phi_n,\phi_m)$
but not for all.  

\subsection{Application to number field tabulation}
\label{intronf}
 The example of this  section provides a convenient illustration 
 of the application of tame-wild inequalities to number field tabulation.    
 The right inequality of \eqref{specificbounds} globalizes to the divisibility relation
 $\fd_{K_5/F} \mid \fd_{K_6/F}$ which on the level of
  magnitudes becomes
 \begin{equation}
 \label{56norm}
 |\fd_{K_5/F}| \leq |\fd_{K_6/F}|.
 \end{equation}
 Consider the problem of finding sextic field extensions $K_6/F$ with 
Galois group either $PSL_2(5)$ or $PGL_2(5)$.   These all
arise as Cayley-Weber resolvents of $K_5/F$ with 
Galois group either $A_5$ or $S_5$.   From \eqref{56norm}, 
one sees that to find all $K_6/F$ with $|\fd_{K_6/F}| \leq B$ it suffices
to find all $K_5/F$ with $|\fd_{K_5/F}| \leq B$, apply 
the Cayley-Weber resolvent, and keep those $K_6/F$ with $|\fd_{K_6/F}| \leq B$.
 This indirect quintic method is enormously faster
for large $B$, but the direct sextic method over $F=\Q$ was used in \cite{FP} and \cite{FPDH} for the
$PSL_2(5)$ and $PGL_2(5)$ cases respectively.

\section{Character theory and discriminants}
\label{chartheory}
In this section, we review how Artin characters underlie discriminants.  
Each of the subsections introduces concepts and notation which
play an important role in the rest of the paper.   The notions 
we emphasize are slightly different from the most standard
representation-theoretic notions.  However they are appropriate here
because all our characters are rational-valued.   

\subsection{Class sets}  
\label{setting1}
   Let $G$ be a finite group.   We say that two elements of $G$ are {\emm power-conjugate}
if each is conjugate to a power of the other.  Let $G^\sharp$ be the set of 
power-conjugacy classes.    Thus one has a natural surjection 
$G \rightarrow G^\sharp$, with the fiber $C_\sigma \subset G$ 
above $\sigma \in G^\sharp$ being its set of representatives.
   The {\emm order}
$\bar{\sigma}$ of an element
$\sigma \in G^\sharp$ is the order of a representing element $g \in C_\sigma$.  
Similarly the {\emm power} $\sigma^k$ of a class $\sigma$ is the class of ${g}^k$ for any representing 
element ${g} \in C_\sigma$.

   When dealing with explicit examples, we most commonly indicate
an element of $G^\sharp$ by giving its order and an extra identifying label, 
as in e.g.\ $2B$.   To emphasize the role of order, 
we say that a class $\tau$ divides a class $\sigma$ if some power of $\sigma$ is $\tau$.  
Thus divisibility of classes $\tau \mid \sigma$ implies divisibility of
integers $\bar{\tau} \mid \bar{\sigma}$, but not conversely.  In connection
with divisibility, the quantity $\left[ \sigma \right] = |C_\sigma|/\phi(\bar{\sigma})$ 
 is useful, 
with $\phi(n) = |(\Z/n)^\times|$ the Euler $\phi$-function.  This quantity
is integral because  $C_\sigma$ consists of $\left[ \sigma \right]$ power-classes,
 each of size $\phi(\bar{\sigma})$.   Alternatively, one
 can think of $G^\sharp$ has indexing conjugacy classes
 of cyclic subgroups of $G$, and then $[\sigma]$ is the 
 number of cyclic subgroups of type $\sigma$.  

Sections~\ref{universalC} and \ref{universalN} systematically
reason with class sets using diagrams based on the divisibility
relation and the quantities $\left[ \sigma \right]$.   
In general, $G$ itself often recedes into the background of our considerations
and the focus is on $G^\sharp$ and its inherited structures.

\subsection{Characters}
\label{setting2}
Our calculations take place mainly in the ring $\Q(G^\sharp)$ of $\Q$-valued
functions on $G^\sharp$.    We write everything out instead using the larger ring
$\R(G^\sharp)$ of real-valued functions, so that we can use standard
terms like cone, hull, and interval with their usual meaning.  
  We make extensive use of the natural 
inner product on $\Q(G^\sharp)$, given by
\[
(f_1,f_2) = \sum_{\sigma \in G^\sharp} \frac{|C_\sigma|}{|G|} f_1(\sigma) f_2(\sigma).
\]
Important elements in  $\Q(G^\sharp)$ for us
include the characters $\phi_X$ of $G$-sets $X$.  
By definition, these characters are obtained
by counting fixed points: $\phi_X(\sigma) = |X^g|$,
for $g$ any representative of $\sigma$.   
Both the identity class $e \in G^\sharp$ and the constant 
function $1 \in \Q(G^\sharp)$ usually play trivial 
roles in our situation.
To efficiently remove these quantities from our attention, 
we define $G^{\sharp 0} = G^\sharp -\{e\}$ and let 
$\Q(G^\sharp)^0 \subset \Q(G^\sharp)$ be the 
orthogonal complement to $1$.   

\subsubsection*{The characters $\phi_{G/H}$ and $a_H$} 
Let $H$ be a subgroup of $G$.  Then the character of
the $G$-set $G/H$ is given by 
\begin{equation}
\label{permchar}
\phi_{G/H}(\sigma) = \frac{|G| |C_\sigma \cap H| }{|H| |C_\sigma|}.
\end{equation}
Taking $H=\{e\}$ gives the regular
character $\phi_G$ with value $|G|$ at $e$ and
$0$ elsewhere.   We define the formal Artin character 
of $H$ to then be the difference
\begin{equation}
\label{formalartin}
a_H = \phi_{G}-\phi_{G/H},
\end{equation}
which lies in $\Q(G^\sharp)^0$.  
Here we use the adjective ``formal'' because often
one talks about Artin characters only in the presence
of fields, while currently we are in a purely
group-theoretic setting.

\subsubsection*{The case that $H$ is cyclic}  
The case where $H$ is cyclic is particularly important to us.  
The generators of $H$ all represent the same class
$\tau \in G^\sharp$ and we use the alternative notation  
$a_\tau = a_H$, calling the $a_\tau$ tame characters for reasons
which will be clear shortly in \S\ref{artinchars}.

 To study the $a_\tau$ explicitly,  it is convenient to make use of
  what we call
 {\emm precharacters} $\hat{a}_\tau$, for $\tau \in G^\sharp$.  
By definition, $\hat{a}_e$ is the $0$ function and 
otherwise $\hat{a}_\tau$ has two nonzero values:
 \begin{align} 
 \label{ahatchar}
\hat{a}_\tau(e) & = |G|, &  \hat{a}_\tau(\tau) & = -\frac{|G|}{|C_\tau|}.
  \end{align}
Tame characters and precharacters are related to each other via
\begin{align}
   \label{aandahat} a_\tau & =  \sum_{k \mid \bar{\tau}} \frac{\phi(\bar{\tau}/k)} {\bar{\tau}} \hat{a}_{\tau^k},  &
   \hat{a}_\tau  & =
    \sum_{k \mid \bar{\tau}} \frac{\bar{\tau} \mu(k)}{\phi(\bar{\tau}) k} a_{\tau^k},
\end{align}
where $\mu$ is the M\"obius $\mu$-functions taking values in $\{-1,0,1\}$.
Thus, $a_e = \hat{a}_e = 0$, 
$a_\tau = \frac{\bar{\tau}-1}{\bar{\tau}} \hat{a}_{\tau}$ if $\bar{\tau}$ is
prime, and otherwise $a_\tau$ and $\hat{a}_\tau$ are not proportional to 
each other.  As $\tau$ ranges over $G^{\sharp 0}$, the $\hat{a}_\tau$ clearly
form a basis for $\Q(G^\sharp)^0$.  So the
$a_\tau$ also form a basis for $\Q(G^\sharp)^0$

\subsection{Artin characters} 
\label{artinchars}
Let $F$ be a local or global number field.
 Let $L/F$ be a Galois extension with Galois group  identified with a subgroup of $G$.  
A permutation representation $\rho$ of $G$ gives an $F$-algebra $K$ 
split by $L$.   For $\fp$ a prime ideal of $F$, the discriminant exponent $c_\fp(K)$ depends 
only on the character $\phi \in \Q(G^\sharp)$ of $\rho$ and in fact 
depends linearly on $\phi$.  The associated Artin character $a_{L/F,\fp}$ is the unique
element of  $\Q(G^\sharp)$ such that one has the general formula 
\begin{equation}
c_\fp(K) = (a_{L/F,\fp},\phi).
\end{equation}
From $c_\fp(F)=0$, one gets $(a_{L/F,\fp},1)=0$ and so $a_{L/F,\fp} \in \Q(G^\sharp)^0$.  
One can completely compute $a_{L/F,\fp}$ by computing
$c_\fp(K)$ for any $|G^{\sharp0}|$ different $K$ having characters
which are linearly
independent in $\Q(G^\sharp)/\Q$.    

Before continuing, we note a subtlety that disappears in
the Artin character formalism that we are reviewing.  Namely
it can happen that non-isomorphic algebras $K'$ and $K''$ give rise 
to the same permutation character $\phi$.  In this case $K'$ and $K''$ 
are called arithmetically equivalent.  
They are indeed equivalent from the point of view of this paper, and any occurrence of $K'$ can
simply be replaced by $K''$.

 An Artin conductor $a_{L/F,\fp}$ can be expressed directly in terms 
of inertia groups in their upper numbering as follows.  
  Let $\fP$ be 
a prime of $L$ above $\fp$ and let $I_{L/F,\fP} \subseteq \gal(L/F) \subseteq G$ be the corresponding
inertia group.  Then one has rational numbers $1\leq s_1 < s_2 < \cdots < s_k$
and normal subgroups   
 \begin{equation}
 \label{inertialchain}
 I_{L/F,\fP} =  I^{s_1} \supset I^{s_2} \supset \cdots \supset I^{s_k}  \supset  \{e\}
 \end{equation}
 satisfying
 \begin{equation}
 \label{Iexpand}
 a_{L/F,\fp} = \sum_{i=1}^k (s_{i}-s_{i-1}) a_{I^{s_i}}.
 \end{equation} 
 Here, for the sake of the conciseness of formulas, we put $s_0=0$.   As a similar
 convention, we put $I^{s_{k+1}}=\{e\}$.   The upper numbers
 $s_i$ we are using here are called {\em slopes} in \cite{jr-local-database} and
 are designed to capture tame and 
 wild ramification simultaneously; one has $s_i=u_i+1$ where
 the $u_i$ are the upper numbers used in the standard reference \cite{serre-CL}.

     If $s_1=1$ then $I^{s_1}/I^{s_2}$ is cyclic of order prime to $p$.  Otherwise,
 all the $I^{s_u}/I^{s_{u+1}}$ are abelian groups of exponent $p$.   
 In particular, $I_{L/F,\fP}$ itself is a {\emm $p$-inertial} group in the
 sense that it is an extension of a prime-to-$p$ cyclic group
 by a $p$-group.   In general, we say that a group is {\emm inertial}
 if it is $p$-inertial for some prime $p$.

    The prime $\fp$ is unramified in $L/F$ if and only if $k=0$ in which case 
$a_{L/F,\fp}$ is zero.  The cases where $\fp$ is ramified but only tamely
 are those with $k=1$ and $s_1=1$.
In both these two settings, \eqref{Iexpand} becomes $a_{L/F,\fp} = a_\tau$ 
with $\tau$ being the class of any generator of any $I_{L/F,\fP}$.   Thus
the tame characters of the previous subsection are exactly the
Artin characters which arise when ramification is tame.

\subsection{Bounds on Artin characters}
\label{bounds}
Define cones in $\R(G^\sharp)^0$ spanned by characters or precharacters as follows:
\begin{alignat*}{2}
\mbox{the tame cone } && T_+(G) & = \langle a_\tau \rangle, \\
\mbox{the wild cone } && W_+(G) & = \langle a_{L/F,\fp} \rangle, \\
\mbox{the inertial cone } && \tilde{T}_+(G) & = \langle a_I \rangle, \\
\mbox{the broad cone } && \hat{T}_+(G) & = \langle \hat{a}_\tau \rangle. 
\end{alignat*}
The tame and broad cones are the simplest of these objects, as
their generators are indexed by the small and explicit set $G^{\sharp 0}$.  
The inertial cone is also a purely group-theoretic 
object, although now more complicated as its generators
are indexed by conjugacy classes of inertial subgroups.  Finally
$W_+(G)$ is much more complicated in nature: its definition
depends on the theory of $p$-adic fields, with
$a_{L/F,\fp}$ running over all possible Artin characters,
as above.  

Our considerations in this section have established the following
inclusions:
\begin{equation}
\label{mainchain}
T_+(G) \subseteq W_+(G) \subseteq \tilde{T}_+(G) \subseteq \hat{T}_+(G).
\end{equation}
The first inclusion holds because tame
characters are special cases of Artin characters, the second by the
expansion \eqref{Iexpand}, and the third because 
all $a_I$ take only positive values on $G^{\sharp 0}$.

 \section{The tame-wild principle} 
 \label{tamewild}
We begin in \S\ref{abstract} by giving a formulation of the tame-wild principle in a somewhat abstract
 context, so that its motivation and structural features can be seen 
 clearly.  Next, \S\ref{proofmethods} observes that the  bounds from the previous section give
 techniques for group-theoretically proving instances of the tame-wild principle.
Finally, \S\ref{calculations} details one way of introducing coordinates
  to render everything explicit and  
  \S\ref{alternative} sketches
  alternative approaches. 

 \subsection{Abstract formulation} 
 \label{abstract} We seek settings
 where general ramification is governed by tame ramification. 
 The statement that equality holds in $T_+(G) \subseteq W_+(G)$ is
 true for some $G$, in which case it is the ideal statement. 
  For general $G$, we seek weaker statements 
 in the same spirit.  
Accordingly, consider 
the orthogonal projection
$a \mapsto a^V$ from $\R(G^\sharp)$ onto an arbitrary 
subspace $V \subseteq \R(G^\sharp)$.  
Let $T_+(G,V)$, $W_+(G,V)$, $\tilde{T}_+(G,V)$,
and $\hat{T}(G,V)$ be the images of 
 $T_+(G)$, $W_+(G)$, $\tilde{T}_+(G)$,
and $\hat{T}(G)$ respectively.  

\begin{defn}
Let $G$ be a finite group and let $V \subseteq \R(G^\sharp)$ 
be a subspace.   If equality holds in $T_+(G,V) \subseteq W_+(G,V)$,
then we say the tame-wild principle holds for $(G,V)$.  
\end{defn}

\noindent  As $V$ gets larger, the tame-wild
principle for $(G,V)$ becomes a stronger statement.  If it holds when
$V$ is all of  $\R(G^\sharp)$, then we say it holds
universally for $G$.  
 
An important aspect of our formalism is as follows.  Given
$(G,V)$, consider inertial subgroups $I$ of $G$.  For each $I$, 
one has the subspace $V_I \subseteq \R(I^\sharp)$ 
consisting of pullbacks of functions in $V$ under
the natural map $I^\sharp \rightarrow G^\sharp$.  
Then the tame-wild
principle holds for $(G,V)$ if and only if 
it holds for all
$(I,V_I)$.   In fact, while $G$ typically arises as a global
Galois group in our applications, 
whether or not the tame-wild principle holds
for $(G,V)$ is purely a question about local
Galois extensions. 
 
\subsection{Two proof methods} 
\label{proofmethods}  Projection turns the chain \eqref{mainchain}
into a chain of cones in $V$:
\begin{equation}
\label{mainchain2}
T_+(G,V) \subseteq W_+(G,V) \subseteq \tilde{T}_+(G,V) \subseteq \hat{T}_+(G,V).
\end{equation}
As we will see, for many $G$ all three inclusions are strict in the universal
case $V = \R(G^{\sharp})$.  However strict inclusions can
 easily become equalities after projection, giving us 
 elementary but quite effective proof techniques.  
 Namely the {\em broad method} for proving
 that the tame-wild principle holds for $(G,V)$ is 
 to show that equality holds in  
 $T_+(G,V) \subseteq \hat{T}_+(G,V)$. 
 The {\em inertial method} is 
 to show that equality holds in
  $T_+(G,V) \subseteq \tilde{T}_+(G,V)$.
  
  Applying the broad method gives the following
  simple result which we highlight because of
  its wide applicability:
\begin{thm} \label{broadprime}  
 Let $G$ be a finite group and let $V$ be
a subspace of $\R(G^\sharp)$.    Suppose 
that the broad cone $\hat{T}_+(G,V)$ is
generated by the $\hat{a}^V_\tau$ with $\tau$ of prime order. 
Then $T_+(G,V) = \hat{T}_+(G,V)$ and the
tame-wild principle holds for $(G,V)$.  
 \end{thm}
 \begin{proof} For $\tau$ of prime order one has $\hat{a}_{\tau} =  \frac{\bar{\tau}}{\bar{\tau}-1} a_\tau$ 
 as noted after \eqref{aandahat}.  Thus $\hat{T}_+(G,V)$ is contained
 in $T(G,V)$ and so all four sets in \eqref{mainchain2} are the same.   
 \end{proof}

\noindent  In general, the broad method is very easy to apply, 
while the harder inertial method can work when the broad method
does not.

\subsection{Calculations with permutation characters}  
\label{calculations}
Let $\phi_1$, \dots, $\phi_r$ be permutation characters spanning $V$.
  Then we are exactly in  the situation described
in the introduction, and in this subsection we describe how one approaches
the tame-wild principle in this particular coordinatization.   We incorporate
the $\phi_i$ into our notation in straightforward ways, for example by
writing $(G,\phi_1,\dots,\phi_r)$ rather than $(G,V)$.  

Throughout this subsection,
we illustrate the generalities by returning to the introductory
example with $G = S_5$ and 
$V = \langle \phi_5,\phi_6 \rangle$.    The very simple
two-dimensional picture of $V$ in Figure~\ref{fivetosix} 
serves as an adequate model for 
mental images of the general situation.   In particular,
we always think of the $a^V_\tau$, $a_{L/F,\fp}$, $a^V_I$, 
and $\hat{a}_\tau$ as in the drawn 
$V$.   We think of our various cones in the drawn $V$ as well.
On the other hand, it is not useful to draw the $\phi_i$ in
on these pictures.  Rather, via the identification of 
$V$ with its dual by the inner product, we think
of the $\phi_i$ as coordinate functions on the
drawn $V$.  

\subsubsection*{Conductor vectors}
     The space $V$ is identified with a subspace of $\R^r$, viewed but not always written as column vectors, 
 via $v \mapsto (c_1,\dots,c_r)$
with $c_i = (v,\phi_i)$.  For example, an Artin character $a^V_{L/F,\fp}$ becomes 
a vector of conductors as in the introduction:
\[
c_{L/F,\fp} = (c_\fp(K_1),\dots,c_\fp(K_r)).
\]
The main case is when the $\phi_i$ are linearly independent, so that $V$ is all of 
$\R^r$.   One can always work in this main case by picking a basis 
from among the $\phi_i$.  

\subsubsection*{Various matrices} 
Our approach to calculations centers on matrices.  
The $r$-by-$G^{\sharp 0}$ {\emm partition
matrix} $P(G,\phi_1,\dots,\phi_r)$
has $i$-$\tau$ entry the cycle type $\lambda_\tau(\phi_i)$ of
$\rho_i(g)$, where $\rho_i$ is a permutation 
representation with character $\phi_i$ and $g \in G$
represents $\tau$.     Thus, 
\begin{equation}
  \label{partmat56}
  P(S_5,\phi_5,\phi_6) = 
  \left(
  \begin{array}{llllll}
  2111 &221&311& 41 & 5 & 32 \\
  222&2211&33& 411 & 51 & 6 \\
  \end{array}
  \right).
  \end{equation}
Partition matrices are purely group-theoretic objects, but one can use fields
in a standard way to help construct them.  For example, the columns from
left to right are the partitions obtained by factoring the pair $(f_5(x),f_6(x))$ from
\eqref{f51}-\eqref{f61} modulo
the primes 67, 211, 31, 13, 11, and 7 respectively.

One passes to the {\emm tame matrix}
$T(G,\phi_1,\dots,\phi_r)$ by replacing each partition
$\lambda_\tau(\phi_i)$ by its conductor $c_\tau(\phi_i) = (a_\tau,\phi_i)$,
 thus
its degree minus its number of parts. Thus
 \begin{equation}
  \label{tamemat56}
  T(S_5,\phi_5,\phi_6) = 
  \left(
  \begin{array}{llllll}
1 &2&2&3 & 4& 3 \\
3 &2&4&3 & 4 & 5 \\
  \end{array}
  \right).
  \end{equation}
The {\emm broad matrix} $\hat{T}(G,\phi_1,\dots,\phi_r)$ 
consists of what we call {\emm preconductors},  the preconductor
 $\hat{c}_\tau(\phi_i) = (\hat{a}_\tau,\phi_i)$ 
being the degree of $\lambda_\tau(\phi_i)$ 
minus its number of ones.  Thus
 \begin{equation}
  \label{pretamemat56}
  \hat{T}(S_5,\phi_5,\phi_6) = 
  \left(
  \begin{array}{llllll}
2 &4&3&4& 5& 5 \\
6 &4&6&4 & 5 & 6 \\
  \end{array}
  \right).
  \end{equation}
 {\emm Inertial matrices} $\tilde{T}(G,\phi_1,\dots,\phi_r)$ 
 typically have more columns, because
 columns are indexed by conjugacy classes
 of inertial subgroups $I$.  But an entry is just the
 {\emm formal conductor} $c_I(\phi_i) = (a_I,\phi_i)$, this 
 being the degree of $\rho_i$ minus the 
 number of orbits of $\rho_i(I)$, just as in 
 the cyclic case.    The cones
 $T_+ \subseteq \tilde{T}_+ \subseteq \hat{T}_+ \subset \R^r$ 
 are then generated by the columns of the corresponding
 matrices $T$, $\tilde{T}$, and $\hat{T}$. 
 
 \subsubsection*{Inclusions $a^V \in T_+(G,V)$ in matrix terms} 
 By dropping rows, we can assume that $\phi_1,\dots,\phi_r$ spans $V$ 
 and so $T = T(G,\phi_1,\dots,\phi_r)$ has full rank $r$, as discussed above.  
 In general,  Let $c \in \R^r$ be a column
 $r$-vector.    For each $r$-element subset $J \subseteq G^{\sharp 0}$ for which
 the corresponding minor $T(J)$ is invertible, let $u(J) = (u(J)_\tau)_{\tau \in J}$ be the vector 
 $T(J)^{-1} c$.  Then $c = \sum_{\tau \in J} u(J)_\tau T_\tau$, with $T_\tau$ the $\tau^{\rm th}$ column
 of $T$.   Then $c$ is in the tame cone $T_+$ if and only if
 there exists such a $J$ with $u(J)_\tau \geq 0$ for all $\tau \in J$.  
 
 To prove that tame-wild holds for $(G,\phi_1,\dots,\phi_r)$ directly, one
 would have to show this positivity condition holds for all conductor
 vectors $c_{L/F,\fp}$.  To show it via the inertial method, one has to show
 that it holds for all formal conductor vectors $c_I$.  To show it
 holds via the broad method, one has to show that it holds for 
 all preconductor vectors $\hat{c}_\tau$.   
  
 \subsubsection*{Projectivization}  
 In the introductory example, we emphasized
 taking ratios of conductors, thereby removing the phenomenon that wild conductors 
 are typically much larger than tame conductors, but keeping the 
 phenomenon we are interested in.  We can do this in the general case as well,
 assuming without loss of generality that $\phi_r$ comes from a faithful
 permutation representation so that the conductors $c_\tau(\phi_r)$ are strictly positive for all
 $\tau \in G^{\sharp 0}$.     We projectivize
 $c = (c_1,\dots,c_r)$ to $c' = (c_1',\dots,c_{r-1}')$ with $c_i' = c_i/c_r$.  
 
 Applying this projectivization process to columns gives the 
 projective tame, inertial, and broad matrices
 respectively, each notationally indicated by a ${}'$. 
 In our continuing introductory example, one has, very simply,
 \begin{eqnarray}
 \label{tprime} T'(S_5,\phi_5,\phi_6) & = & \left( \begin{array}{cccccc} 1/3 & 1 & 1/2 & 1 & 1 & 3/5 \end{array} \right), \\
\label{hattprime}  \hat{T}'(S_5,\phi_5,\phi_6) & = & \left( \begin{array}{cccccc} 1/3 & 1 & 1/2 & 1 & 1 & 5/6 \end{array} \right).
 \end{eqnarray}
 In general, the $\tau$-columns of $T'(G,\phi_1,\dots,\phi_r)$ and 
 $\hat{T}'(G,\phi_1,\dots,\phi_r)$ agree if $\tau$ has prime order.   Here
 they disagree only in the last column corresponding to the composite 
 order $6$.  
 
 Let $T'_+(G,\phi_1,\dots,\phi_r)$ be the convex hull of the columns of $T'(G,\phi_1,\dots,\phi_r)$
 and define $W'_+(G,\phi_1,\dots,\phi_r)$, $\tilde{T}'_+(G,\phi_1,\dots,\phi_r)$ and $\hat{T}'_+(G,\phi_1,\dots,\phi_r)$ 
 to be the analogous hulls.   Then \eqref{mainchain2} has its obvious analog
 at the level of hulls, and one can think about the broad method and the inertial method
 at this level.  In the introductory
 example, \eqref{tprime} and \eqref{hattprime} say that
 $T_+'(S_5,\phi_5,\phi_6)  \subseteq \hat{T}_+'(S_5,\phi_5,\phi_6)$ is
 equality because both sides are $[1/3,1]$.  Thus the tame-wild principle
 holds for $(S_5,\phi_5,\phi_6)$.  
  
 The drop in dimension from $r$ to $r-1$ has a number of advantages.  As illustrated already by 
 \eqref{tprime}-\eqref{hattprime},
 it renders the $r=2$ case extremely concrete.  As we will illustrate in \S\ref{tauandI}, it renders the
 $r=3$ case highly visible.   In general, it lets one determine the truth of $a^V \in T'_+(G,\phi_1,\dots,\phi_r)$ 
 by computation with $(r-1)$-by-$(r-1)$ minors rather than $r$-by-$r$ minors.  
 
\subsection{Alternative approaches}  
\label{alternative} 
Our abstract formulation of the tame-wild principle is designed 
to be very flexible.  For example, say that a vector $v \in \R(G^\sharp)$
is {\emm bad} if $(a_\tau,v) \geq 0$ for all $\tau \in G^{\sharp 0}$ but
$(a_{L/F,\fp},v)<0$ for some Artin character $a_{L/F,\fp}$.  
The bad vectors form a union of cones in $\R(G^\sharp)$
and the tame-wild principle holds for $(G,V)$ if and only if $V$
misses all these cones.    In this sense, the one-dimensional
$V$ spanned by bad $v$ are essential cases, but
these $V$ are never spanned by permutation 
characters.  

Sections~\ref{universalC} and \ref{universalN} are in the universal setting 
$V = \R(G^\sharp)$ and we do not use $\phi_i$ at all.
Sections~\ref{comparing} and \ref{general} 
 return to the permutation
character setting described in \S\ref{calculations}.  
In general, the systematic
study of the tame-wild principle for a
fixed $G$ and varying $V$ 
would be facilitated by the canonical 
basis of $\Q(G^\sharp)$ given by
irreducible rational characters.

\section{The universal tame-wild principle holds for U-groups}
\label{universalC} In
\S\ref{gsharp}, we present a diagrammatic way of understanding class sets $G^{\sharp}$.
Making use of this viewpoint, \S\ref{expansionformal} gives the canonical expansion of 
a formal Artin character $a_I$ as a sum of tame characters $a_\tau$.  
Next, \S\ref{applying} introduces the notion of U-group and proves that
the universal tame-wild principle holds for U-groups.    
However the class of U-groups is quite small, as discussed in 
\S\ref{Cclass}.

\subsection{Divisibility posets} 
\label{gsharp}
    For $G$ a finite group, the set $G^\sharp$ is naturally a partially ordered set via the divisibility relation.    
 We draw this divisibility poset in the standard way with an edge from $\sigma$ down
 to $\tau$ of vertical length one if $\sigma^p=\tau$ for some prime $p$.   With notations as in 
 \S\ref{setting1}, the natural weight $d(\tau,\sigma) = \left[\sigma\right]/\left[\tau\right]$
 plays an important role, and we write it next to the edge whenever it
 is different from $1$, considering this data as part of the divisibility poset.      
 
     The product of the edge weights from any vertex $\sigma$ down to 
another $\tau$ is path-independent, being in fact just $d(\tau,\sigma) = \left[\sigma\right]/\left[\tau\right]$.  
Define integers $u_{G,\sigma}$ via
\begin{equation}
\label{dsum}
\sum_{\tau|\sigma} d(\tau,\sigma) u_{G,\sigma} = 1.
\end{equation}
Thus $u_{G,\tau} = 1$ for maximal $\tau$ and all the integers $u_{G,\tau}$ can be
computed by downwards induction on the divisibility poset $G^\sharp$.   
 
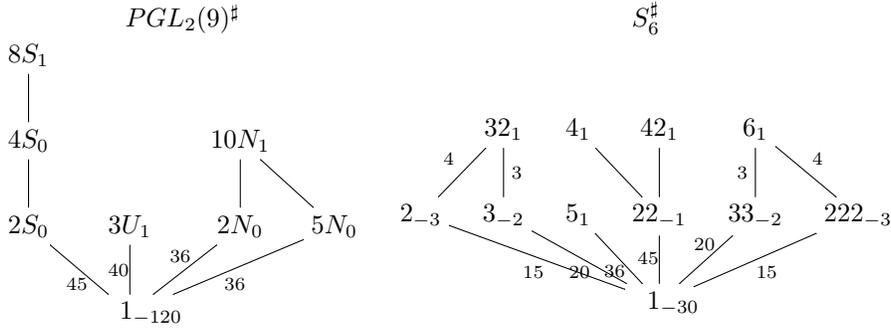
\begin{figure}[htb]
\[
\begin{array}{cc}
PGL_2(9)^\sharp & S_6^\sharp \\
\xymatrix@R=1.5pc@C=.8pc{
8S_1 \ar@{-}[d] \\
4 S_0 \ar@{-}[d]  & &  10N_1 \ar@{-}[d]  \ar@{-}[dr] \\
2S_0 \ar@{-}[dr]_{45 \!\!\!\!} & 3U_1\ar@{-}[d]_{40 \!\!}  & 2N_0 \ar@{-}[dl]_{36\!\!\!} &  5 N_0 \ar@{-}[dll]^{\!\!36}  \\
& \;\;\;\;1_{-120} \!\!\!&& 
}
&
\xymatrix@R=1.5pc@C=.8pc{
&\\
& 32_1 \ar@{-}[dl]_4 \ar@{-}[d]^3 & 4_1 \ar@{-}[dr]& 42_1 \ar@{-}[d]& 6_1 \ar@{-}[d]_3 \ar@{-}[dr]^4 & \\
2_{-3} \ar@{-}[drrr]_{15} & 3_{-2} \ar@{-}[drr]_{20 \!\!\!} & 5_1 \ar@{-}[dr]_{36\!\!\!\!} & 22_{-1} \ar@{-}[d]_{45\!\!} & 33_{-2}
 \ar@{-}[dl]_{20 \!\!\!\!} & 222_{-3} \ar@{-}[dll]^{15} \\
&&& \;\; 1_{-30} \!\!\! &&\\
}
\end{array}
\]
\caption{\label{classspaces} Two divisibility posets $G^{\sharp}$ with
$u_{G,\tau}$ subscripted on $\tau \in G^{\sharp}$.  }
\end{figure}

    Figure~\ref{classspaces} draws the divisibility posets 
    $PGL_2(9)^\sharp$ and $S_6^\sharp$, with
 each $\tau$ subscripted by its $u_{G,\tau}$.       The case
 $PGL_2(9)$ represents the general case $PGL_2(p^f)$,
 with split-torus classes indexed by non-unital divisors of 
 $p^f-1$, a unipotent class $pU$, non-split-torus classes
 indexed by non-unital divisors of $p^f+1$, and finally the identity
 class $1$.   The case $S_6$ represents the general case
 $S_n$, where classes are indexed by partitions of $n$, with
 $1$'s usually left unprinted. 

 In general, the largest edge weights on divisibility posets
 $G^\sharp$ tend to be on edges incident on the identity
 class.  These edges do not play an important role for us 
 and in the sequel we work instead with 
 the divisibility poset associated to $G^{\sharp0}$.  
 
\subsection{Expansion of formal Artin characters}
\label{expansionformal}
Divisibility posets for inertial groups $I$ and the associated integers
$u_{I,\sigma}$ are important to us because of the role they play in the following lemma.
\begin{lemma}
\label{expansionformallemma}
Let $G$ be a finite group, let $I$ be a subgroup, and let $i : I^\sharp \rightarrow G^\sharp$ be the induced map.  
Then the expansion of the formal Artin character $a_I \in \Q(G^\sharp)^0$ in 
the basis $\{a_\tau\}_{\tau \in G^{\sharp 0}}$ can be read off from the divisibility 
poset $I^{\sharp0}$ via the formula
\begin{equation}
\label{aIexpand}
a_I = \frac{1}{|I|} \sum_{\sigma \in I^{\sharp 0}} u_{I,\sigma} [\sigma] \bar{\sigma} a_{i(\sigma)}.
\end{equation}
\end{lemma}
\noindent Before proving the lemma, we explain the roles that various parts of  \eqref{aIexpand}
play in the sequel.  The positive integer $[\sigma] \overline{\sigma}$ plays a very passive role: 
only the positivity of $[\sigma] \overline{\sigma}$ is used in the proof
of Theorem~\ref{easy-cgroup}; moreover, $[\sigma] \overline{\sigma}$ factors out
in Lemma~\ref{artinexpand} and accordingly does not enter into
 \S\ref{pqr}-\S\ref{pqrpos}.    The factor $|I|^{-1}$ is more important: 
while only its positivity enters into the proof of Theorem~\ref{easy-cgroup}, it contributes to the index 
factor in \eqref{wformula} which enters significantly into the rest of Section~\ref{universalN}.   
The part with the most important
role is $u_{I,\sigma}$, as it is the possible negativity of $u_{I,\sigma}$ that can lead to
failures of the tame-wild principle.   Our use of the function $i$ relegates the difference between
$I$ and $G$ to the background, but one should note that for $\tau \in G^{\sharp 0}$ 
the actual coefficient of $a_\tau$ in \eqref{aIexpand} has $|i^{-1}(\tau)|$ terms.  
\begin{proof}  First consider the case $I=G$.  Then both sides of \eqref{aIexpand} are in 
$\Q(I^\sharp)^0$.  The left side takes the value
$a_I(\tau) = - 1$ for all $\tau \in I^{\sharp 0}$.  We thus need to evaluate the
right side on an arbitrary $\tau \in I^{\sharp 0}$ and see that it simplifies to $-1$.   Calculating,
\begin{eqnarray*}
\frac{1}{|I|} \sum_{\sigma \in I^{\sharp 0}} u_{I,\sigma} [\sigma] \bar{\sigma} a_\sigma(\tau) & = & \frac{1}{|I|}
\sum_{\tau \mid \sigma} u_{I,\sigma} [\sigma] \bar{\sigma} a_\sigma(\tau) \\
& = & \frac{1}{|I|} \sum_{\tau \mid \sigma} u_{I,\sigma} [\sigma] \bar{\sigma}  \left( - \frac{\phi(\bar{\tau})}{\bar{\sigma}} \frac{|I|}{|C_{\tau}|} 
\right) \\
&=& - \sum_{\tau \mid \sigma} u_{I,\sigma} [\sigma] \frac{\phi(\bar{\tau})}{|C_{\tau}|}  \\
& = & - \sum_{\tau \mid \sigma} \frac{u_{I,\sigma} [\sigma]}{[\tau]}  \\
& = & - \sum_{\tau \mid \sigma} d(\tau,\sigma) u_{I,\sigma} \\
& = & - 1.
\end{eqnarray*}
Here we have used formulas from \S\ref{setting1} and \S\ref{setting2} as well
as the definition of $d(\tau,\sigma)$ and the defining property of the $u_{I,\sigma}$
 from \S\ref{gsharp}.  Finally the case of general
$G$ follows, by induction of both sides from $I$ to $G$.  
\end{proof}

\subsection{Applying the inertial method}
\label{applying}
     Say that a class $\tau \in G^\sharp$ is a {\emm U-class} if it divides exactly
 one maximal element $\sigma$ of $G^\sharp$ and $d(\tau,\sigma)=1$. 
   Otherwise, say it is an {\emm N-class}.  
 Here U stands for unique and N for nonunique.
  The following three facts are immediate from the definition.
  First, a maximal class $\tau$ is always a U-class with $u_{G,\tau}=1$.  
   Second, other U-classes $\tau$ have $u_{G,\tau}=0$.  Third, a maximal N-class 
   $\tau$ always has $u_{G,\tau} < 0$.     
   
      We divide all finite groups into two types, as follows. 
   \begin{defn}
  A finite group is a {\emm U-group} if every non-identity element is
  contained in exactly one maximal cyclic subgroup.  Otherwise
  it is an {\emm N-group}.  
\end{defn}
\noindent It is immediate that a group $G$ is a U-group if and only if all
classes $\tau \in G^{\sharp 0}$ are U-classes.   
Thus from Figure~\ref{classspaces}, $PGL_2(9)$ is a U-group while
$S_6$ is an N-group.    It follows easily from the definition
that any subgroup of a U-group is itself a
U-group.  Similarly, any quotient of a U-group is
a U-group \cite{suzuki1}.

Via \eqref{aandahat}, the chain~\eqref{mainchain} 
completely collapses to the equality $T_+(G) = \hat{T}_+(G)$ 
if and only if all non-identity elements in $G$ have prime order.   
The following theorem is a subtler version of this idea.  
   
\begin{thm} \label{easy-cgroup} Suppose $G$ is a group such that all inertial 
subgroups of $G$ are U-groups.  Then one has $T_+(G)=\tilde{T}_+(G)$ and
so the tame-wild principle 
holds universally for $G$.  
\end{thm}

\begin{proof}   Let $I$ be an arbitrary inertial subgroup.  Since $I$ is assumed
to be a U-group, the associated integers $u_{I,\sigma}$ are nonnegative
for all $\sigma \in I^{\sharp 0}$.    
For any $\tau \in G^{\sharp 0}$, the terms  
$|I|^{-1} u_{I,\sigma} [\sigma] \bar{\sigma}$
contributing to the coefficient of $a_\tau$ in Lemma~\ref{expansionformallemma} are all nonnegative.  Hence
the coefficient itself is nonnegative and so $a_I$ is in the 
tame cone $T_+(G)$.   Since the $a_I$ generate the 
inertial cone $\tilde{T}_+(G)$, equality holds in
$T_+(G) \subseteq \tilde{T}_+(G)$.
 \end{proof}
 
\noindent In particular, the tame-wild principle holds for all U-groups.  This is the main import
of Theorem~\ref{easy-cgroup}, as we are not aware of any 
group satisfying the hypothesis of 
Theorem~\ref{easy-cgroup} which is not itself
a U-group.

\subsection{Classification of U-groups}
\label{Cclass}

    Given Theorem~\ref{easy-cgroup}, it is of interest to classify U-groups.  
This problem has been addressed in the literature with 
Kontorovich \cite{kontorovich1,kontorovich2} referring
to U-groups as completely decomposable groups, and Suzuki \cite{suzuki1}
calling them groups with a complete partition.   We give a summary
of the classification situation here.

    The condition to be a U-group is very restrictive, but it is easy to
check that it includes many groups of small order.  In particular, the
following groups are U-groups: cyclic groups, dihedral groups, groups
of prime exponent, and the Frobenius groups $F_p = C_p : C_{p-1}$.  
The last class is particularly important in our context, since an extension
of a $p$-adic field of degree $p$ has normal closure with Galois group
a subgroup of $F_p$.  If $q$ is a prime power, the linear groups
$\PSL_2(q)$ and $\PGL_2(q)$ are U-groups, so that in particular
$S_4 \cong \PGL_2(3)$, $S_5 \cong \PGL_2(5)$ and $A_6 \cong \PSL_2(9)$ are
all U-groups.  
There are more U-groups then those listed here, most of them being
more general types of Frobenius groups.

    The following observation is useful in understanding the nature of
U-groups.  In two settings, the extreme members of a class of groups are exactly
the U-groups as follows.  First, consider abelian $p$-groups of order
$p^n$.  Up to isomorphism, they correspond to partitions of $n$.  The
groups which are U-groups are the two extreme ones $(C_p)^n$ and
$C_{p^n}$.  Second, consider semidirect products $C_a :_\gamma C_b$
with $a$ and $b$ being relatively prime and $\gamma : C_b \rightarrow
\mbox{Aut}(C_a)$.  If $\gamma$ is trivial, then $C_a :_\gamma C_b
\cong C_{ab}$ is a U-group.  If $\gamma$ is injective, then $C_a
:_\gamma C_b$ is again a U-group, being of a nature similar to $F_p$
above.  Again, it is the intermediate cases which are N-groups:
if $\gamma$ is neither trivial nor injective then non-trivial elements
in the kernel of $\gamma$ are in more than one maximal cyclic
subgroup.

\section{The universal tame-wild principle usually fails for N-groups} 
\label{universalN}
  In this section,
 we study the universal tame-wild principle for N-groups.
 In \S\ref{artinexpand}, we give the canonical expansion of 
 a general Artin character $a_{L/F,\fp}$ in terms of tame characters $a_\tau$.    
 In  \S\ref{pqr}, we list out N-groups of order $pqr$ where $p$, $q$,
 and $r$ are not necessarily distinct primes, finding 
 six series.  We show in \S\ref{pqrneg} that the universal tame-wild 
 principle generally fails for groups in the first four 
 series.   In \S\ref{pqr5} we take a close look 
 at the quaternion group $Q_8$, which is the first 
 group of the fifth series, finding failure again.  
 On the other hand we show in \S\ref{pqrpos} that the universal tame-wild principle holds for all
 groups in the sixth series.  Finally,  \S\ref{promote} explains how the
 negative results for small groups support the
 principle that most N-groups do not satisfy the
 universal tame-wild principle.

 \subsection{Expansion of general Artin characters}
 \label{artinexpand}
     Let $a_{L/F,\fp} \in \Q(G^{\sharp})^0$ be an Artin character 
 coming from an inertial subgroup $I \subseteq G$.  
 Equation~\eqref{Iexpand} expands $a_{L/F,\fp}$ in terms of formal
 Artin characters $a_{I^{s_i}}$ and Lemma~\ref{expansionformallemma}
  in turn expands each
 $a_{I^{s_i}}$ in terms of tame characters.    Putting these two
 expansions together and replacing the divisibility posets
 $(I^{s_i})^{\sharp0}$ with their images in $I^{\sharp0}$ 
 gives the following lemma.  
 \begin{lemma} 
 \label{artinexpandlemma}Let $G$ be a group, and let $a_{L/F,\fp} \in \Q(G^{\sharp})^0$
 be an Artin character with inertia group $I=I^{s_1} \supset I^{s_2} \supset \cdots$
 as in \eqref{inertialchain}.  Let $i : I^\sharp \rightarrow G^\sharp$ be the induced map.  Then one has 
 the expansion 
  \begin{equation*}
 a_{L/F,\fp} = \frac{1}{|I|} \sum_{\sigma \in I^{\sharp 0}} w_{L/F,\fp,\sigma} [\sigma] \bar{\sigma} a_{i(\sigma)}
 \end{equation*}
 where
 \begin{equation}
 \label{wformula}
w_{L/F,\fp,\sigma} = \sum_{i=1}^k (s_i-s_{i-1}) [I:I^{s_i}] u_{I^{s_i},\sigma}.
\end{equation}
\end{lemma}
\noindent 
While the lemma applies to the general situation, our focus in \S\ref{pqr}-\S\ref{pqrpos} is 
on the case $I = G$.  Here  $a_{L/F,\fp}$ is in the tame
cone $T_+(G)$ if and only if $w_{L/F,\fp,\tau} \geq 0$ for all $\tau \in G^{\sharp 0}$.

 \subsection{Inertial N-groups of order $pqr$}
 \label{pqr}   Groups
 of order $p$ or $pq$ are U-groups.   In the complete list
 of inertial groups of order $pqr$, in a rough sense
 about half of them are U-groups and the other half N-groups.  
 For example, for a given prime $p$, there are two non-abelian groups of
 order $p^3$, the extra-special groups often denoted 
 $p_+^{1+2}$ and $p_-^{1+2}$.   For $p$ odd, $p_+^{1+2}$ has exponent
 $p$ and so is a U-group, while $p_-^{1+2}$ is an N-group.  Similarly
 the dihedral group $D_4 = 2_+^{1+2}$ is a U-group while
 the quaternion group $Q_8 = 2_{-}^{1+2}$ is an N-group.
  
 In fact, it is easy to see that the inertial N-groups are as follows.   Now $p$, $q$, and
 $r$ are required to be different primes, with $q \mid p-1$ whenever $F_{p,q} = C_p:C_q$ is present:
 \begin{description}
 \item[1] The product $F_{p,q} \times C_r$. 
 \item[2] The semi-direct product $C_{p}:C_{q^2} \cong F_{p,q} *_{C_q} C_{q^2}$.
  \item[3] The abelian group $C_{pq} \times C_p \cong C_p \times C_p \times C_q$.
  \item[4] The product $F_{p,q} \times C_p$.
 \item[5] The extra-special group $p_-^{1+2}$.
 \item[6] The abelian group $C_{p^2} \times C_p$.
 \end{description}
 These groups $I$ are all $p$-inertial groups, but not inertial groups for any other primes.   
 Moreover, since all proper subgroups are U-groups, the universal tame-wild principle
 fails for $I$ if and only if there exists a totally ramified local Galois extension $L/F$ having
 $\gal(L/F) \cong I$ with
 associated Artin character $a_{L/F,\fp}$ having a negative coefficient $w_{L/F,\fp,\tau}$.
 Furthermore, in each case it turns out that 
 there is exactly one N-class $\tau \in I^{\sharp0}$.  
 Only for this class $\tau$ could $w_{L/F,\fp,\tau}$ possibly be negative,
 and this N-class
 is boxed in the displayed divisibility posets below.      
 
 In general, let $I$ be a $p$-inertial group.      Then it is known that there 
 indeed exists a totally ramified Galois extension of $p$-adic fields $L/F$
 with $\gal(L/F)$ isomorphic to $I$.   This fact for our 
 particular $I$ is essential to our proofs that the universal tame-wild
 principle does not hold.  However it easy to prove this fact for all
 the above $I$ by direct exhibition of $L/F$.   We will go into this
 level of detail only for the groups in Series~4 and $Q_8$ from  Series~5, as here we need particular fields
 satisfying conditions on their wild ramification.

 \subsection{Negative results for four series}
 \label{pqrneg} 
        Our first result concerns Series $1$-$4$ and is negative:
 \begin{thm}   
 \label{pqrnegthm} $F_{3,2} \times C_3 \cong S_3 \times C_3$ satisfies the universal tame-wild principle,
 but otherwise the groups $F_{p,q} \times C_r$, $C_{p}:C_{q^2}$, $C_{pq} \times C_p$, and  $F_{p,q} \times C_p$,  do not.
 \end{thm}
    
\begin{proof}    In the divisibility posets below,  the wild classes, meaning the classes of $p$-power order,
 are put in boldface for further emphasis.    For the first three series, the unique N-class $\tau$ 
has prime-to-$p$ order and so we do not need to enter into an examination of wild slopes.
In Series~4, $\tau$ has order $p$ and bounds on wild slopes lead to the exception.  

1.  For the group $I = F_{p,q} \times C_r$, 
power-conjugacy classes are determined
by their orders and the divisibility poset $I^{\sharp0}$ is
\begin{equation*}
\xymatrix{
pr \ar@{-}[d]  \ar@{-}[dr]  &  qr \ar@{-}[d]_{p \!} \ar@{-}[dr] & \; \\ 
\mathbf{ p }& \fbox{$r$}   & q.   
}
\end{equation*} 
Equation~\eqref{wformula} becomes $w_r = u_{I,r} = -p<0$.    
 So by 
the existence of totally ramified $I$-extensions as discussed in the previous subsection,
the universal tame-wild principle does not hold for $I= F_{p,q} \times C_r$.  

2.  The group $I = C_p:C_{q^2}$ behaves very similarly.
Again power-conjugacy classes are determined by their
orders:
\[
\xymatrix{
pq \ar@{-}[d] \ar@{-}[dr] & q^2 \ar@{-}[d]^p \\
\mathbf{p}  & \fbox{$q$} \; .  
}
\]
The key quantity  $w_q = u_{I,q} = -p$ is again negative, so the universal tame-wild principle
fails for $C_p : C_{q^2}$.  

3.  The group $I = C_p^2 \times C_q$ has a more complicated 
divisibility poset $I^{\sharp 0}$ but the behavior is otherwise similar.  
The classes of order $p$ and the classes of order $pq$ 
have the structure of projective lines over $\F_p$ in  
bijection with one another:  
\begin{eqnarray*}
\xymatrix{
p_0q \ar@{-}[d]  \ar@{-}[drrrrr]& \cdots & p_i q \ar@{-}[d]  \ar@{-}[drrr]& \cdots & p_\infty q \ar@{-}[d]  \ar@{-}[dr]& \\
\mathbf{p_0} & \cdots & \mathbf{p_i}    & \cdots & \mathbf{p_\infty}  & \fbox{$q$} \; .  
}
\end{eqnarray*}
Once again $w_q=u_{I,q}=-p$ and so the universal tame-wild principle fails
for $C_{p}^2 \times C_q$. 

4.  For $I = F_{p,q} \times C_p$ the divisibility poset $I^{\sharp0}$
is disconnected: 
\[
\xymatrix{
        &    &    &     & p_\infty q \ar@{-}[d]_{p} \ar@{-}[dr] &  \\
\mathbf{p_0} & \cdots & \mathbf{p_i} & \cdots &
\fbox{$\mathbf{p_\infty}$} & q \, .  
}
\]
Here $p_0$ and $q$ lie in the factor $F_{p,q}$ while $p_\infty$ lies in the
factor $C_p$.   The first term in \eqref{wformula} is $u_{I,p_\infty} =1-p$.
However now we must take into account how wild ramification contributes 
to the remaining terms.  Let $s>1$ be the slope associated to $F_{p,q}$ 
and let $c>1$ be the slope associated to $C_p$.   Since $C_p$ is abelian,
$c$ must be integral and hence $c \geq 2$.   On the other hand, 
$s$ must have exact denominator $q$.  
Let $m = \min(c,s)$ so that $I^m = C_p^2$ is the wild inertia group
and $I^{\max(c,s)} \cong C_p$ is a higher inertia group.       
If $c>s$ then $(I^c)^{\sharp 0} = \{p_\infty\}$, while if $s>c$ then
$(I^s)^{\sharp 0}= \{p_0\}$. 
Equation~\eqref{wformula} becomes
\[
w_{p_\infty} = \left\{
\begin{array}{ll} 
(1-p) + q (s-1) + q p (c-s)  & \mbox{if $c > s$,} \\
(1-p) + q (c-1)                     & \mbox{if $s > c$.}
\end{array}
\right.
\]
For $(p,q) = (3,2)$, the general formula
simplifies to
\[
w_{3_\infty} = \left\{
\begin{array}{ll} 
6c-4s-4  & \mbox{if $c > s$,} \\
2c-4     & \mbox{if $s > c$.}
\end{array}
\right.
\]
Thus, using $c \geq 2$, one has $w_{3_\infty} \geq 0$ and so the universal
tame-wild principle holds for $F_{3,2} \times C_3$.    

There are many ways to produce an explicit instance with
$w_{p_\infty}<0$ for the remaining $(p,q)$.  We will present one in
the setting $s > c=2$ in which case $w_{p_\infty} = 1+q-p$ is indeed
negative.  To get an $F_{p,q}$ extension, start with $x^p-p$ which
gives a totally ramified $F_{p,p-1}$ extension of $\Q_p$ with wild
slope best written in the form $1 + \frac{p}{p-1}$.  Write $e =
(p-1)/q$ and extend the ground field from $\Q_p$ to $F_e =
\Q_p[\pi]/(\pi^e-p)$.  Then $x^p-p$ has Galois group $F_{p,q}$ over
$F_e$, with wild slope $1 + \frac{ep}{p-1}$, as tame base-change
always scales slopes this way.  But now $x^p-\pi x^{p-1} + \pi$ has
wild slope $2$ and, after perhaps replacing $F_e$ by an unramified
extension $F$, Galois group $C_p$ \cite{amano}.  The splitting field
of $(x^p-p) (x^p-\pi x^{p-1} + \pi)$ gives the desired extension $L/F$
showing that the universal tame-wild principle does not hold for
$F_{p,q} \times C_p$.
\end{proof}

\subsection{Negative result for $Q_8$} 
\label{pqr5}  
The fifth series, consisting of groups of the form $p_-^{1+2}$, is the
most complicated.  Here we treat only 
 $2_-^{1+2}=Q_8$, getting a negative result.  
\begin{prop} 
\label{q8prop} 
The universal tame-wild principle fails for the
quaternion group.
\end{prop}
\begin{proof}
The divisibility poset $Q^{\sharp 0}_8$, with unique N-class boxed as always, is 
\[
\xymatrix{
4_i \ar@{-}[dr] & 4_j \ar@{-}[d] & 4_k. \ar@{-}[dl]  \\
    & \,\, \fbox{$\mathbf{2}$}.&
    }
\]
The generic case has three distinct slopes.   We seek only
counterexamples and so we focus on the special case 
with two slopes $s_1<s_2$, with $s_1$ occurring
with multiplicity two.  The key quantity \eqref{wformula} 
here becomes  
$w_2 = -2 s_1 + 4 (s_2-s_1) = 4 s_2 - 6 s_1$.  Thus
one gets a counterexample to the universal tame-wild principle
if and only if $s_2<1.5 s_1$.   

  The table of octic 2-adic fields \cite{jr-2adicoctics} available
from the website of \cite{jr-local-database} then give four types of counterexamples in this context, after tame
base-change from $\Q_2$ to its Galois extension $F$ with 
ramification index $t$ and residual
field degree $u$:
\begin{equation}
\label{q8counters}
\begin{array}{ccccccc}
\# & c & \mbox{Slope Content} & I & D & s_1 & s_2  \\
\hline
2 & 10 & [1.\overline{3},1.\overline{3},1.5]_3^2 & SL_2(3) & GL_2(3) & 2 & 2.5\\ 
2 & 16 &  [2,2,2.5]^2 & Q_8 & \hat{Q}_8 & 2 & 2.5  \\
4  & 16 & [2,2,2.5]^4 & Q_8 & 8T17 & 2 & 2.5 \\ 
8  & 22 & [2.\overline{6},2.\overline{6},3.5]_3^2 & SL_2(3) & GL_2(3) & 6 & 8.5  \\
\end{array}
\end{equation}
Here in the first and last cases,  we use the general conversion from slope content $[\cdots \sigma_i \cdots ]_t^u$ 
over $\Q_2$ to slope content $[\cdots s_i \cdots]_1^1$ over $F$ given by
$s_i = 1 + t(\sigma_i-1)$.   A full treatment of the range of possible counterexamples
could have \cite[Prop.~4.4]{fontaine} as its starting point.    \end{proof}

Our counterexamples in \S\ref{elusive} and \S\ref{bestcounter} will be built from one of the
two fields with slope content $[2,2,2.5]^2$.  A point to note here is that $\Q_2$ does  
have totally ramified quaternionic extensions, in fact four of them, all
with slope content $[2,3,4]$ \cite{jr-2adicoctics}.   However these extensions do not give counterexamples
to the universal tame-wild principle for $Q_8$.  The fact that the first local counterexamples 
come from $\hat{Q}_8 = 8T8$ extensions of $\Q_2$ plays a prominent
role in our later global counterexamples.

\subsection{Positive results for $C_{p^2} \times C_p$}   
\label{pqrpos}
Here we prove that the N-groups in Series~6 always satisfy the universal tame-wild principle.  
Unlike most of our previous positive results, but like the exception $S_3 \times C_3$ of \S\ref{pqrneg},
this result is not purely group-theoretic.  Rather it depends on
a close analysis of the possibilities for wild slopes.   
Said in a different way, the situation for these $I$ is $T_+(I) = W_+(I) \subset \tilde{T}_+(I)$
so that the universal tame-wild principle holds, even though it is not provable
by the inertial method.  

 \begin{thm}  
 \label{series6thm} The groups $C_{p^2} \times C_p$ satisfy the universal tame-wild principle. 
 \end{thm}

\begin{proof} 
Let $K_{p^2}/F$ be a cyclic extension of degree $p^2$ and slopes $s_1<s_2$.  Let
$K_p/F$ be a cyclic extension of degree $p$ and slope $t$.    Switch to the 
indexing scheme of \cite{serre-CL} via $s_i = 1+ v_i$ and $t = 1 + c$, so as to 
better align also with our reference \cite{fontaine} and in particular make \eqref{geom} below
as simple as possible.  There are three possibilities
for how the slope filtration goes through the group:
\begin{equation*}
\begin{array}{ccccc}
c < v_1 < v_2  &\;\;& v_1 \leq c < v_2 &\;\;& v_1 < v_2 \leq c\\
\xymatrix@C=0.49pc{
\mathbf{p^2} \ar@{-}[d]_p \\
\fbox{$\underline{\mathbf{p_0}}$} & p_1 & \cdots & p_\infty,  \\
} &&
\xymatrix@C=0.49pc{
p^2 \ar@{-}[d]_p  \\
\fbox{$\underline{\mathbf{p_0}}$} 
& \mathbf{p_1} & \mathbf{\cdots} & \mathbf{p_\infty},
}  &&
\xymatrix@C=0.49pc{
p^2 \ar@{-}[d]_p \\
\fbox{$\mathbf{p_0}$} & \mathbf{p_1}  & \mathbf{\cdots} & \underline{\mathbf{p_\infty}}.  } 
\end{array}
\end{equation*}
Here, assuming all inequalities are strict, classes in the higher inertia group of order $p^2$ are put in bold and
the classes in the higher inertia group of order $p$ are furthermore underlined.  If one
has equality, the formulas below still apply.  

As in the previous two subsections, only one $w_\tau$ from Lemma~\ref{artinexpandlemma}
 could possibly be
negative, and in this case it is $w=w_{p_0}$.   Equation~\eqref{wformula} becomes   
\begin{equation}
\label{uform}
w = 
\left\{
\begin{array}{ll}
 (c+1)  (1-p) + (v_2-v_1) p^2 & \mbox{ if  $c < v_1 < v_2$,} \\
 (v_1+1)  (1-p) + (c-v_1) p + (v_2-c)  p^2 & \mbox{ if $v_1 \leq c < v_2$,} \\
 (v_1+1) (1-p) + (v_2-v_1) p & \mbox{ if $v_1 < v_2 \leq c$.}
\end{array}
\right.
\end{equation}
    Let $e$ be the ramification index of $F/\Q_2$ and put $B = e/(p-1)$.  From the known behavior
    of cyclic degree $p$ extensions, one has
   \begin{equation}
  \label{double}
  1 \leq c \leq pB,  \hspace{1in} 1 \leq v_1 \leq pB.
  \end{equation}
  There are two regimes to consider, the geometric regime where
  $v_1 < B$ and the arithmetic regime where $v_1 \geq B$.   
  One has
  \begin{align}
\label{geom}  v_2 & \geq pv_1 & & \hspace{-.8in} \mbox{in the geometric regime}, \\
\label{arith} \mbox{and } v_2 & = v_1 + e && \hspace{-.8in} \mbox{in the arithmetic regime}.
  \end{align}
  These last two facts and other related information dating back to \cite{maus}
  are conveniently available in 
 \cite[Prop.~4.3]{fontaine}.
  
      The quantity $e$ does not enter into the geometric inequality \eqref{geom} and 
  since we need to deal with arbitrary $e$ the upper bounds
 in \eqref{double}  are not available to us.  This fact is
 the source of our terminology because the geometric case is now identified
 with the case where $p$-adic fields have been replaced by $\F_{p^f}((t))$, which
 have $e=\infty$.    The worst case is always when $v_2 = p v_1$ and, in the second
 case, when $c$ takes on its limiting bound $v_2$ as well.  Substituting these worst cases into \eqref{uform} and
 simplifying, one has 
  \[
w \geq
  \left\{ 
  \begin{array}{ll}
 (v_1+1)(1-p) + (pv_1 -v_1)p^2  & \mbox{ if $c < v_1 < v_2$,} \\ 
  (v_1+1)(1-p) + (p v_1-v_1) p  & \mbox{ if $v_1 \leq c < v_2$,} \\
 (v_1+1)(1-p) + (p v_1 - v_1) p  &  \mbox{ if $v_1 < v_2 \leq c$.}  \\
  \end{array}
  \right.
  \]
  With $m$ equal to $p^2$, $p$, $p$ in the three cases. one further simplifies by
  \[
  w \geq (p-1)(-v_1-1+v_1 m) = (p-1)((m-1)v_1-1) \geq 0.
  \]
  Thus in the geometric regime, $w$ is never negative.

   In the arithmetic regime the substitute \eqref{arith} for \eqref{geom} is simpler in that 
   it is an equality, but now the  
 upper bounds in \eqref{double}  will need to be used.  
  The substitution $v_2 = v_1 + e = v_1 + B (p-1)$ into \eqref{uform} makes $w$  
  factor and we divide by the positive quantity $p-1$:
  \[
  \frac{w}{p-1} = 
  \left\{ 
  \begin{array}{ll}
  B p^2 -1-c, & \mbox{ if $c < v_1 < v_1+e$,} \\
  B p^2 - v_1 - p (c-v_1) - 1, & \mbox{ if $v_1 \leq c < v_1+e$,}  \\
  B p - v_1-1, & \mbox{ if $v_1 < v_1+e \leq c$.} 
  \end{array}
  \right.
  \]
  Using the bounds \eqref{double} one has
   \[
  \frac{w}{p-1} \geq 
  \left\{ 
  \begin{array}{ll}
  B p^2 -1-Bp  = B p (p-1) - 1 = e p-1 \geq 1,\\
  B p^2 - B p - p (e - 1)  - 1 = e p - p( e-1) - 1 = p-1 \geq 1,  \\
  B p - v_1-1 \leq B p - (B p-e) - 1 = e-1 \geq 0.
  \end{array}
  \right.
  \]
  in the three cases.  Thus here too $w \geq 0$.    
  \end{proof}

\subsection{From smaller to larger groups} 
\label{promote}
 Our final topic in this section is to promote our counterexamples from the small groups $I$ to larger
groups $G$ that contain them.   In general, let $I \subseteq G$ be an 
inclusion of groups and consider the induced map 
$i : I^\sharp \rightarrow G^\sharp$.  Then the 
lack of injectivity of $i$ can obstruct
the promotion process.   For example, consider Series 4 groups $I = F_{p,q} \times C_p$ 
and their product embedding into $G=S_{p^2}$.  Then all $p+1$ classes in $I^\sharp$ of
order $p$ go to the single class in $S_{p^2}^\sharp$ indexed by the partition
$p^p = p \cdots p$.  To get the coefficient of $a_{p^p}$ of the pushed-forward formal Artin conductor 
$a_I \in \Q(S_{p^2}^\sharp)^0$ one has to add the contributions of the fiber, as in 
Lemma~\ref{expansionformallemma}.  There are $p$ contributions of $1/pq$ and one contribution of $(2-p)/pq$ for a total
of $2/pq$.   Equation~\ref{wformula} says that wild ramification 
can only increase this $2/pq$ to larger positive numbers,  and so all pushed-forward
Artin characters $a_{L/F,\fp}$ from $I$ are in the tame cone $T_+(S_{p^2})$.

For Series 1-3 and also for $Q_8$, this complication does 
not arise because the unique N-class in
$I^{\sharp0}$ is the only class of its order.  Hence the promotion process works:
\begin{cor}  Let $G$ be a group containing a subgroup $I$ of the form $F_{p,q} \times C_r$,
$C_p:C_{q^2}$, $C_{pq} \times C_p$ or $Q_8$.  Then $G$ does not satisfy the universal tame-wild
principle.  
\end{cor}
\noindent Since there are so many possibilities for $I$, the hypothesis holds for many $G$.  Moreover
the fact that it holds for a given $G$ is often easily verified.  For example, when studying $G$ one commonly has a list of maximal subgroups $H$, and one can often easily see that
at least one $I$ is in at least one of the $H$.  As another example, the presence 
of $C_{pq} \times C_p$ can often be read off from the divisibility
poset: suppose one has a class $\tau \in G^\sharp$ of order $\bar{\tau} = pq$ not dividing
a class of order $p^2 q$ but such that $p^2$ divides the numerator of  $|G|/|C_\tau|$.  
Then any representative $g$ of $\tau$ 
lies in a group of type $C_{pq} \times C_p$.   
This criterion is satisfied particularly often for $p=2$ and
some odd prime $q$.

\section{Comparing an algebra with its splitting field}
\label{comparing}
     In this section we return to a very concrete setting, considering
 types $(G,\phi_i,\phi_r)$ 
where $\phi_i$ comes from a faithful permutation representation
$i : G \subseteq S_n$ and $\phi_r$ is the regular character.  Thus we are
considering algebras $K=K_i$ of a specified Galois type
compared with their splitting fields $K^{\g} = K_r$.  

    In \S\ref{generalities} we introduce explicit notation for comparing two algebras
and in 
\S\ref{meanroot} we explain how it is sometimes best to highlight
root discriminants $\fr$ rather than discriminants $\fd$.   
The tame-wild principle in the notation set up then takes the following form:
\[
\fr_{K^{\g}/F}^{\underline{\alpha}(G,\phi_i,\phi_r)} \mid \fr_{K/F} \mid \fr_{K^\g/F}^{\underline{\omega}(G,\phi_i,\phi_r)}.
\]
We observe in \S\ref{right} that the right divisibility often trivially holds.
In \S\ref{elusive}, we give four examples where it holds non-trivially 
and one where it fails to hold.  In \S\ref{left} we show that the left divisibility
always holds, and discuss applications to number field tabulation.

 \subsection{Generalities} 
 \label{generalities}
 The case $r=2$ of just two algebras deserves special attention for at least 
three reasons.  First, hulls $T'_+(G,\phi_1,\phi_2) \subset \R^1$ are intervals while 
hulls for larger $r$ can have up to $|G^{\sharp 0}|$ vertices.   
Second, the inequality for each face of any  $T'_+(G,\phi_1,\dots, \phi_r)$ 
also comes from some $T'_+(G,\psi_1,\psi_2)$ with the new characters $\psi_j$
being certain sums of the old characters $\phi_i$.  Third,
it is the case which applies most directly to number field tabulation.

To present results coming from $r=2$ as explicitly as possible, we 
 let $\alpha = \alpha(G,\phi_1,\phi_2)$ and $\omega = \omega(G,\phi_1,\phi_2)$ 
be the left and right endpoints of the interval $T'_+(G,\phi_1,\phi_2)$.  
The tame-wild principle says that all local exponents satisfy
\begin{equation}
\label{addinequal}
\alpha c_{\fp}(K_2) \leq c_{\fp}(K_1) \leq \omega c_{\fp}(K_2).
\end{equation}
In this $r=2$ setting, the tame-wild principle breaks cleanly
into two parts: the left and right tame-wild principles 
respectively say that the left and right inequalities in
\eqref{addinequal} always hold.   Similarly, one
has the perhaps larger inertial interval $[\tilde{\alpha},\tilde{\omega}]$
and the perhaps even larger broad interval
$[\hat{\alpha},\hat{\omega}]$. 

To transfer the additive inequalities \eqref{addinequal} into the multiplicative language of
divisibility, we make use of the following formalism. 
Note that the torsion-free group $\mathcal{I}$ of fractional ideals of a local or global number field 
$F$ embeds into 
its tensor product over $\Z$ with $\Q$, a group we write as $\mathcal{I}^\Q$ to account for the fact that
$\mathcal{I}$ is written multiplicatively.    In $\mathcal{I}^\Q$, as our notation indicates, general rational exponents  on ideals 
are allowed.   Then \eqref{addinequal} corresponds to 
\begin{equation}
\label{multinequal}
\fd_{K_2/F}^\alpha \mid \fd_{K_1/F} \mid \fd_{K_2/F}^\omega,
\end{equation}
which makes sense for both local and global number fields.   In this
formalism, the relations of the introductory example take the form
$\fd_{K_6/F}^{1/3} \mid \fd_{K_5/F} \mid \fd_{K_6/F}$.

\subsection{Mean-root normalization and the comparison interval}  
\label{meanroot} It is sometimes insightful to switch to a slightly different
normalization.  We call this normalization mean-root normalization, with ``mean'' capturing
how additive quantities are renormalized and ``root'' capturing how multiplicative quantities
are renormalized.  

   If $K/F$ has degree $n$ and discriminant $\fd_{K/F}$ then
its {\emm root discriminant} is by definition $\fr_{K/F}=\fd_{K/F}^{1/n}$.    To
make this shift in our formalism, we simply 
replace all permutation characters $\phi_i$ by the scaled-down
quantities 
$\underline{\phi}_i = \phi_i/\phi_i(1)$.   One has mean tame conductors 
$\underline{c}_\tau(\phi_i) = (a_\tau,\underline{\phi}_i)$ as well as their analogs 
 $\underline{c}_I(\phi_i) = (a_I,\underline{\phi}_i)$ and 
$\underline{\hat{c}}_\tau(\phi_i) = (\hat{a}_\tau,\underline{\phi}_i)$.     
We always indicate this alternative convention by
underlining.  Thus the mean-root normalized tame hull for two characters
indexed by dimension is $\underline{T}'_+(G,\phi_n,\phi_m) = [\underline{\alpha},\underline{\omega}]$ 
where $\underline{\alpha}=m \alpha/n$ and  $\underline{\omega}=m \omega/n$.
The divisibility relation~\eqref{multinequal} becomes 
$\fr_{K_m/F}^{\underline{\alpha}} \mid \fr_{K_n/F} \mid \fr_{K_m/F}^{\underline{\omega}}$.  

    The {\emm comparison interval} $[\underline{\alpha},\underline{\omega}]$ just discussed
    supports an intuitive 
understanding of how ramification in $K_n/F$ and $K_m/F$ relate 
to each other.  Suppose, for example, that $K_m/K_n/F$ is a tower of fields so that one
has the standard divisibility relation
\begin{equation}
\label{rootdiscdivide}
\fr_{K_n/F} \mid \fr_{K_m/F}.
\end{equation}
Then, assuming $K_m/F$ is actually ramified,
 the ratio $\log(|\fr_{K_n/F}|)/\log(|\fr_{K_m/F}|) \in [0,1]$ can 
 be understood as the fraction of ramification in $K_m/F$ which
 is seen already in $K_n/F$.    If the corresponding
 tame-wild principle holds, then this quantity is guaranteed
 to be in $[\underline{\alpha},\underline{\omega}]$.   
 
     The mean-root normalization introduces a sense of 
absolute scale, with the number one playing a prominent role, 
as illustrated by the preceding paragraph and 
the next three subsections.   Assuming $\phi_n-\phi_m$ is
not a constant, one always has strict inequality 
$\underline{\alpha}<\underline{\omega}$.   
The failure of resolvent constructions from $(G,\phi_n)$-fields to $(G,\phi_m)$-fields to preserve ordering
by absolute discriminants is roughly speaking measured
by the size of $[\underline{\alpha},\underline{\omega}]$.  
For $\phi_n$ and $\phi_m$ coming from faithful transitive
permutation representations, a very common situation is
$\underline{\alpha} \leq 1 \leq \underline{\omega}$.   
This tendency gets stronger as $n$ and $m$ increase
to $|G|$.  For example, for $(A_5,\phi_{20},\phi_{30})$ 
the partition matrix is 
$\left( \begin{array}{ccc} 2^{10} & 3^{6} 1^2 & 5^4 \\2^{14} 1^2 &  3^{10} & 5^{6} \end{array} \right)$
and the comparison interval works out to $[9/10,15/14] \approx [0.90,1.07]$.

\subsection{The right tame-wild principle often holds for $(G,\phi_i,\phi_r)$}
\label{right}
Applying \eqref{rootdiscdivide}  in our setting gives 
\begin{equation}
\label{chain}
 \fr_{K/F} \mid \fr_{K^\g/F}
\end{equation}
when $K$ is a field.  
This relation holds also when $K$ is an algebra,
as can be seen by 
expressing $K$ as a 
product of fields and comparing each factor to the field $K^\g$.

The critical quantity is simply expressed as 
\begin{equation}
\label{omegaform}
\underline{\omega}(G, \phi_i, \phi_r) = 
\max_{\tau \in G^{\sharp 0}} \frac{\underline{c}_\tau(\phi_i)} {\underline{c}_\tau(\phi_r)}.
\end{equation}
The denominator depends only on the order $\bar{\tau}$ of $\tau$ via 
$\underline{c}_\tau(\phi_r) = (\bar{\tau}-1)/\bar{\tau}$.  
For the more complicated numerator, one has
$\underline{c}_\tau(\phi_i) \leq (\bar{\tau}-1)/\bar{\tau}$, with
 equality if and only if the partition $\lambda_\tau(\phi_i)$ has
the form $\bar{\tau}^{n/\bar{\tau}} = \bar{\tau} \cdots \bar{\tau}$. 
A permutation is {\emm semiregular} if all cycles have
the same length. Therefore,
$\underline{\omega}(G, \phi_r, \phi_i) \leq 1$, with equality if and only if $G$
contains a non-identity element which is semiregular.   Summarizing:
\begin{prop} 
\label{righttw}
Let $G \subseteq S_n$ be a permutation group containing a non-identity
semiregular element, $\phi_i$ the
given permutation character, and $\phi_r$ the regular character. Then 
the right tame-wild principle holds for $(G,\phi_i,\phi_r)$ with
\[
\underline{\omega}(G,\phi_i,\phi_r) = 1.
\]
However this principle is nothing more than the classical statement
that for any $(K,K^\g)$ of type $(G,\phi_i,\phi_r)$, one
has  $\fr_{K/F} \mid \fr_{K^\g/F}$.
\end{prop}

\subsection{Elusive groups} 
\label{elusive}
In the global setting, we are mainly interested in the case when 
$K$ is a field and thus $G$ is transitive.   A transitive permutation group 
which does not contain a non-identity
semiregular element is called an {\emm elusive group} 
\cite{elusive}.
So Proposition~\ref{righttw} 
is the best statement for non-elusive transitive groups, but the 
situation needs to be investigated further for elusive 
groups.   

  Elusive groups are aptly named in that they are relatively rare. The
smallest $n$ for which $S_n$ contains an elusive group is $n=12$.  There 
are five elusive groups in $S_{12}$, listed in Table~\ref{elusivetab}, all subgroups of the
 Mathieu group $M_{11}$ in its transitive degree twelve realization
12T272.   Here and in the sequel we use the $T$-notation for transitive permutation
groups introduced in \cite{conway-hulpke-mckay} and available online
in several places, including \cite{LMFDB}.  

 The following proposition treats these five groups.  
\begin{prop} \label{elusiveprop} The right tame-wild principle for $(G,\phi_i,\phi_r)$ holds for the elusive groups $12T46$, $12T84$,
$12T181$, and $12T272$ with $\underline{\omega}(G,\phi_i,\phi_r) = 20/21$.  Thus
the strengthening $\fr_{K/F} \mid \fr_{K^{\rm gal}/F}^{20/21}$ of \eqref{omegaform} holds
for these groups.     For $12T47$, one has $\underline{\omega}(12T47,\phi_i,\phi_r)=8/9$.
Extensions $(K_i,K_r)$ from \eqref{counter47} give an counterexample
to the tame-wild principle over $\Q(\sqrt{-3})$, but there is no  
counterexample
over $\Q$.  
\end{prop}
\begin{proof}

The part below the line of Table~\ref{elusivetab} supports applying the
broad and inertial methods  for $12T272 \cong M_{11}$.  
Thus the line labeled $\tau$ lists out the seven elements of $12T272^{\sharp0}$.   The  
next two lines gives the corresponding dodecic partitions $\lambda_\tau(\phi_i)$ 
and conductors $c_\tau(\phi_i) = 12 \underline{c}_\tau(\phi_i)$ respectively.  
The next lines give 
$\underline{c}_\tau(\phi_r)=c_\tau(\phi_r)/|M_{11}|$ and
$\underline{c}'_\tau = \underline{c}_\tau(\phi_i)/\underline{c}_\tau(\phi_r)$.  
Thus the comparison interval is $[\underline{\alpha},\underline{\omega}] = [2/3,20/21]$.  

\begin{table}[htb]
{\renewcommand{\arraycolsep}{2.5pt}
\[
\begin{array}{rl|ccccccc|cc}
12T46 & \cong C_3^2:Q_8&  \surd & \surd & \surd &  &  & \surd &  \\
12T47 & \cong M_9  &  \surd & \surd & \surd &  &  &  &  \\
12T84 & \cong C_3^2:\hat{Q}_8 &   \surd & \surd & \surd &  & \surd & \surd &  \\
12T181  & \cong M_{10} & \surd & \surd & \surd & \surd &  & \surd &  \\
12T272  & \cong M_{11} & \surd & \surd & \surd & \surd & \surd  & \surd & \surd  \\
\hline
& \tau & 2A & 3A & 4A & 5A & 6A & 8A & 11A & Q_8 & I \\
&\lambda_\tau(\phi_{i})  & 2^4 1^4 & 3^3 1^3 & 4^2 2^2 & 5^2 1^2 & 6321 & 84 & (11) 1 & 84    \\
&c_\tau(\phi_{i}) & 4 & 6 & 8 & 8 & 8 & 10 & 10 & 10 & \leq 10 \\
&\underline{c}_\tau(\phi_r)  & 1/2 & 2/3 & 3/4 & 4/5 & 5/6 & 7/8 & 10/11 & 7/8 & (|I|-1)/|I|  \\
&\underline{c}'_\tau   & 2/3 & 3/4 & 8/9 & 5/6 & 4/5 & 20/21 & 11/12 & 20/21& \leq 20/21
\end{array}
\]
}
\caption{\label{elusivetab} Information used in the proof of Proposition~\ref{elusiveprop} }
\end{table}

One inertial subgroup of $12T272$ is $Q_8$, which has orbit partition $84$.   
As indicated in the second-to-last column of Table~\ref{elusivetab}, its associated quantity
is $\underline{c}'_{Q_8}= 20/21$, which is the right endpoint of $[2/3,20/21]$.
In general, the difficulty with the inertial method is that there can be many inertial N-subgroups $I$
to inspect.  However here we can treat them all at once as follows.  Since none of the
elusive groups from the complete list are themselves inertial groups, $I$ must
act intransitively and so $c_I(\phi_i) \leq 10$.   Also $\underline{c}_I(\phi_r) = (|I|-1)/|I| \geq 7/8$,
since N-groups have order at least~$8$.  So, as indicated by the table, 
$\underline{c}'_I \leq 20/21$.   Thus our initial case $Q_8$ was in fact the worst case, 
and the right tame-wild principle holds for $(M_{11},\phi_i,\phi_r)$.    

The part of the table above the line gives the partitions which arise for all the $G$,
as a subset of those that we have listed for $11T272$.    The smaller groups 
$12T46$, $12T84$, and $12T181$ still have elements of cycle type 84, and so 
the same argument goes through for them, proving the tame-wild principle
for $(G,\phi_i,\phi_r)$ in these cases.  Note that our uniform treatment
of all the $G$ uses that $\underline{c}_\tau(\phi_r)$ is independent
of $G$; in contrast, the unnormalized quantity $c_\tau(\phi_r)$ depends 
on $G$.    

A counterexample seems likely for $G=12T47$ because it contains $Q_8$ and its comparison interval
is only $[\underline{\alpha},\underline{\omega}] = [2/3,8/9]$, with $8/9$ being considerably less than $20/21$.  
 However, as discussed in reference to \eqref{q8counters}, there are no quaternionic extensions
of $\Q_2$ giving counterexamples to the universal tame-wild principle for $Q_8$.   Other candidate
$I$ do not work either, and we are forced to leave $\Q$ as a ground field.  

Our counterexample comes from fields $K_n=\Q[x]/f_n(x)$ with discriminants $D_n \in \Z$ and 
Galois groups $G_n = \gal(K_n^\g/\Q)$ as follows:  
\begin{equation}
\label{counter47}
\begin{array}{r | lrl}
n & f_n(x) & D_n & G_n  \\
\hline
8 &                  x^8+6 x^4-3 & - 2^{16} 3^7 & \hat{Q}_8  \\
9&               x^9-3 x^8+18 x^5+18 x^4-27 x+9 &  -2^{16} 3^{15} & 9T19  \\
12&         x^{12}-6 x^{10}-4 x^9+12 x^7-36 x^5+30 x^4+8 x^3-8      & \;\;\; 2^{22} 3^{18} & 12T84 
\end{array}
\end{equation}
The overgroup $12T84 \cong C_3^2:\hat{Q}_8 \supset 12T47$ was chosen because it contains
not just $Q_8$ but also $\hat{Q}_8$. 
The nonic group  $9T19$ is a lower
degree realization
of $12T84$, where the isomorphism with $ C_3^2:\hat{Q}_8$ is naturally 
realized. 
The field $K_8$ was chosen as a strong candidate from which to build a counterexample,
because $\gal(K_8^\g /\Q)$ is its own decomposition group 
 with slope content $[2,2,2.5]^2$ as in \eqref{q8counters}.   The field $K_9$ 
was extracted from the database \cite{jr-global-database} as a $9T19$ field
with $K_8$ as a resolvent,  and then $K_{12}$ was obtained from $K_9$ by
resolvent calculations.   

The splitting field $K_{12}^\g$ contains $\Q(\sqrt{-3})$ with $\gal(K_{12}^\g/\Q(\sqrt{-3}))=12T47$
by construction.  The root discriminant of $K_{12}^\g$ is 
$2^{2} 3^{127/72}$, as computed by using the website of \cite{jr-local-database} to analyze 
ramification in $K_9/\Q$.    Here the exponent $2$ can be confirmed from a standard computation
associated with the slope content $[2,2,2.5]^2$, namely $2/8+2/4+2.5/2=2$.  On the other
hand $K_{12}$ has root discriminant $2^{22/12} 3^{18/12}$.  The quotient $(22/12)/2 = 11/12$
is to the right of the root-normalized tame hull $[\underline{\alpha},\underline{\omega}]=[2/3,8/9]$,
giving a counterexample to the right tame-wild principle for $(12T47,\phi_i,\phi_r)$ over
$\Q(\sqrt{-3})$.  
\end{proof}

\subsection{The left tame-wild principle always holds for $(G,\phi_i,\phi_r)$}
\label{left}

The following theorem shows that an  important part of the tame-wild principle holds for all 
finite groups $G$.  

\begin{thm}  Let $G \subseteq S_n$ be any permutation group, $\phi_i$ the
given permutation character, and $\phi_r$ the regular character.   Let $\fix{G}$ be 
the maximal number of fixed points of a non-identity element of $G$.  Then 
the left tame-wild principle holds for $(G,\phi_i,\phi_r)$ with
\[
\underline{\alpha}(G,\phi_i,\phi_r) = 1 - \frac{\fix{G}}{n}.  
\]
Thus for any $(K,K^{\rm gal})$ of type $(G,\phi_i,\phi_r)$, 
one has $\fr_{K^{\g}/F}^{1-\fix{G}/n} \mid \fr_{K/F}$.
\end{thm}

\begin{proof}  We apply the broad method. 
Let $\tau$ be an arbitrary element of $G^{\sharp0}$ 
and call its order $t$.   Consider 
the corresponding column ${\left( \!\!\! \begin{array}{c} { \lambda} \\ { \Lambda} \end{array} \! \!\! \right)}$
 of the partition matrix $P(G,\phi_i,\phi_r)$. 
Then $\lambda = \prod_{k=1}^t k^{m_k}$ is some partition of $n$ and   
$\Lambda = t^{|G|/t}$ is the corresponding partition of $|G|$.  

The projective matrices $T'(G,\phi_i,\phi_r)$ and $\hat{T}'(G,\phi_i,\phi_r)$ have just one
row each.  The entries in the $\tau$ column are respectively $c(\lambda)/c(\Lambda)$ and
$\hat{c}(\lambda)/\hat{c}(\Lambda)$.  Their difference is positive:
   \begin{eqnarray*}
 \frac{\hat{c}(\lambda)}{\hat{c}(\Lambda)} -  \frac{c(\lambda)}{c(\Lambda)}   & = &  \frac{\sum_{k=2}^t m_k k}{|G|}- \frac{\sum_{k=2}^t m_k (k-1)}{|G| (t-1)/t} \\
&=& \frac{1}{|G|} \sum_{k=2}^t m_k \frac{k(t-1)}{t-1}-\frac{1}{|G|} \sum_{k=2}^t m_k \frac{t(k-1)}{t-1}\\
& = & \frac{1}{|G|} \sum_{k=2}^t m_k \frac{(k t-k) - (k t - t)}{t-1} \\
& = & \frac{1}{|G|} \sum_{k=2}^t m_k \frac{t-k}{t-1} \\
& \geq & 0.
  \end{eqnarray*}
  Thus $\hat{c}(\lambda)/\hat{c}(\Lambda) \geq c(\lambda)/c(\Lambda)$ and so
  certainly all the $\hat{c}(\lambda)/\hat{c}(\Lambda)$ are at least 
  \[
  \alpha(G,\phi_i,\phi_r) = \min_{\tau \in G^{\sharp 0}} \frac{c_\tau(\lambda)}{c_\tau(\Lambda)}.
  \]
  Thus the left tame-wild principle holds for $(G,\phi_i,\phi_r)$.  
  
  Since the left tame-wild principle holds for $(G,\phi_i,\phi_r)$, we can compute $\alpha(G,\phi_i,\phi_r)$ 
  using  
  $\hat{c}_\tau$ rather than $c_\tau$, giving 
  \[
  \alpha(G,\phi_i,\phi_r) = \min_{\tau \in G^{\sharp 0}} \frac{\hat{c}_\tau(\lambda)}{\hat{c}_\tau(\Lambda)} = \min_{\tau \in G^{\sharp0}} \frac{n-m_1}{|G|} =  \frac{n-\fix{G}}{|G|}.
  \]
  Switching to the mean-root normalization gives $\underline{\alpha}(G,\phi_i,\phi_r)=1 - \fix{G}/n$. 
  \end{proof}

\subsubsection*{Number field tabulation}  For certain solvable transitive groups $G \subset S_n$, 
the techniques
of \cite{jj-wallington} let one compute all degree $n$ fields $K$ of type $G$
where $|\fr_{K^\g/\Q}|$ is at most some constant $\beta$.   Then 
the theorem just proved can be applied through
its corollary  
$|\fr_{K^{\g}/\Q}^{1-\fix{G}/n}| \leq |\fr_{K/\Q}|$ to
obtain all $K$ with  $|\fr_{K/\Q}|$ at most $B = \beta^{1-\fix{G}/n}$.  
This computation is 
carried out in \cite{jj-nonic} for the primitive nonic groups
$9T9$, $9T14$, $9T15$, $9T16$,
  $9T23$, and $9T27$ to obtain the corresponding nonic fields with 
  smallest absolute discriminant.    This particular application
  served as the catalyst for the present paper.

\section{Examples and counterexamples}
\label{general}  

  The positive and negative results of the previous sections give one a good 
idea of the extent to which the tame-wild principle holds and how it 
can be applied.  We now refine this picture, by considering various $(G,\phi_1,\dots,\phi_r)$ of interest
and determining whether the tame-wild principle holds.
In \S\ref{tauandI}, we give examples illustrating the broad method and the inertial method. 
In \S\ref{bestcounter}, we conclude by arguing that counterexamples to the
tame-wild principle from pairs $(K_1,K_2)/\Q$ of number fields are not easily found, 
but present one such counterexample with Galois group $12T112$ 
of order $192$.  

\subsection{The broad and inertial methods}  We illustrate the two methods of \S\ref{proofmethods}
with
positive results for three N-groups.  
\label{tauandI}
    
    \subsubsection*{The broad method for $(\mbox{\rm Aff}_3(\F_2), \phi_7,\phi_8,\phi_{8a},\phi_{8b})$}
    \label{aff3}
 The group $\mbox{Aff}_3(\F_2)$ provides a simple illustration of the 
 broad method in the setting $r=4$.  It has five non-trivial small permutation 
 representations $\rho_{7a}$, $\rho_{7b}$, $\rho_{8}$, $\rho_{8a}$, 
 $\rho_{8b}$, with images the permutation groups $7T5$, $7T5$, $8T37$, $8T48$, 
 $8T48$.   
 The first three representations are 
 through the quotient $GL_3(\F_2) \cong PGL_2(\F_7)$ while the last two
 are faithful.   
  The representations
  $\rho_{7a}$ and $\rho_{7b}$ share a common character $\phi_7$.  They are thus 
  arithmetically equivalent and we call them identical twins.  The representations
  $\rho_{8a}$ and $\rho_{8b}$ have different characters $\phi_{8a}$ and
  $\phi_{8b}$ and so we call them fraternal twins.  
  The four characters $\phi_7$, $\phi_8$, $\phi_{8a}$,
  and $\phi_{8b}$ are linearly independent.

  \begin{figure}[htb]
  \[
\begin{array}{c|ccccccccc}
\tau &    2A & 2B & 2C & 3A & 4A & 4B & 4C & 6A & 7A \\
 \hline
   &    1^7 & 22111 & 22111 & 331 & 22111 & 421 & 421 & 331 & 7 \\
 P &    1^8 & 2222 & 2222 & 3311 & 2222 & 44 & 44 & 3311 & 71 \\
  &    2222 & 221111 & 2222 & 3311 & 44 & 4211 & 44 & 62 & 71 \\
  &    2222 & 2222 & 221111 & 3311 & 44 & 44 & 4211 & 62 & 71 \\
\hline
           &  0 & 4 & 4 & 6 & 4 & 6 & 6 & 6 & 7 \\
\hat{T} &                  0 & 8 & 8 & 6 & 8 & 8 & 8 & 6 & 7 \\
+ &                  8 & 4 & 8 & 6 & 8 & 6 & 8 & 8 & 7 \\
 &                  8 & 8 & 4 & 6 & 8 & 8 & 6 & 8 & 7 \\
                   \hline
   &                 0 & 2 & 2 & 4 & 2 & 4 & 4 & 4 & 6 \\
T &                0 & 4 & 4 & 4 & 4 & 6 & 6 & 4 & 6 \\
\bullet  &                4 & 2 & 4 & 4 & 6 & 4 & 6 & 6 & 6 \\
  &                 4 & 4 & 2 & 4 & 6 & 6 & 4 & 6 & 6 \\
\multicolumn{10}{c}{\;}
   \end{array}
\]
\includegraphics[width=4in]{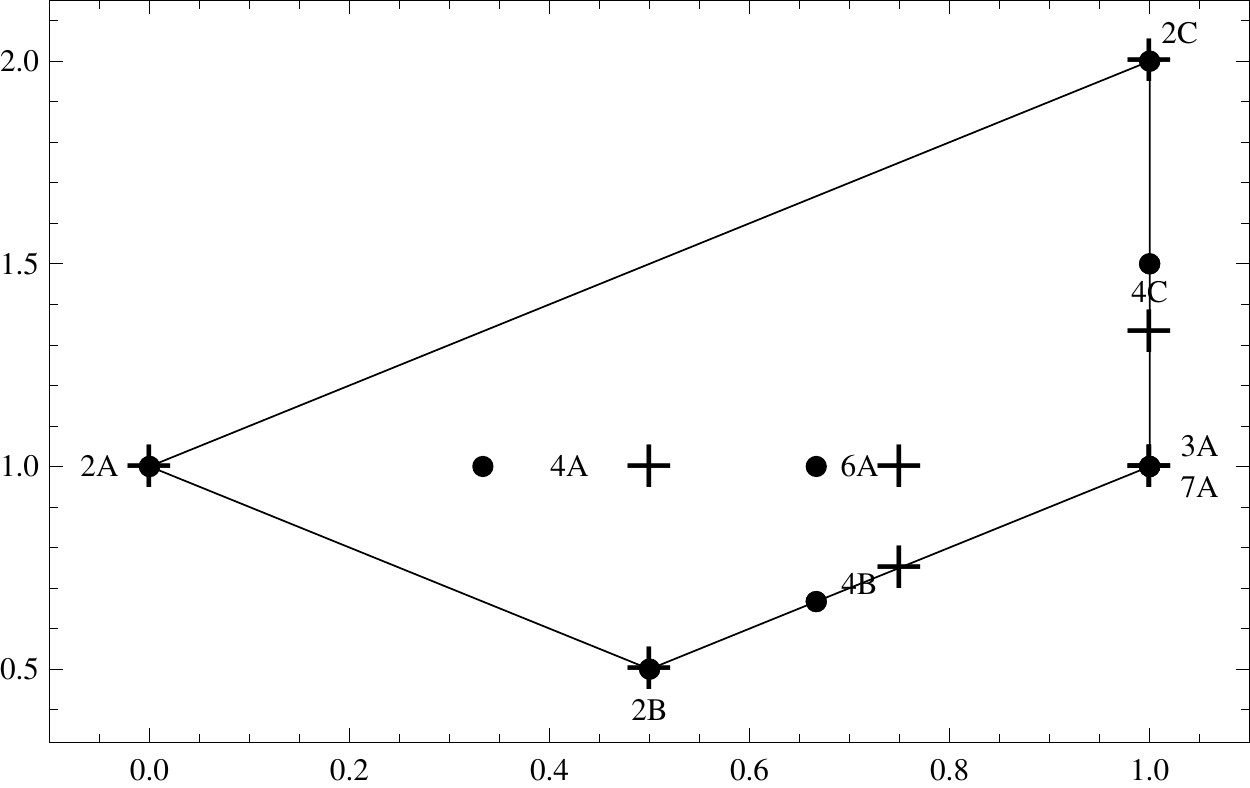}
\caption{\label{pictg1344}  Top: The partition matrix, broad matrix, and tame matrix
for $(\mbox{Aff}_3(\F_2),\phi_7,\phi_8,\phi_{8a},\phi_{8b})$.   Bottom: the broad hull and
tame hulls coinciding after removing $\phi_8$ from consideration, proving the tame-wild principle for 
$(\mbox{Aff}_3(\F_2),\phi_7,\phi_{8a},\phi_{8b})$.
}
\end{figure}

Figure~\ref{pictg1344} first presents the partition matrix 
  $P = P(\mbox{Aff}_3(\F_2), \phi_7,\phi_8,\phi_{8a},\phi_{8b})$ and 
  the broad and tame matrices derived from it.   
  For visualization purposes, it then drops consideration of $\phi_8$.
    After this projection, it 
  plots the columns of $\hat{T}'$ as $+$'s and those of ${T}'$ 
  as $\bullet$'s.  Since the $+$'s are in the hull ${T}'_+$ of the
  $\bullet$'s, the tame-wild principle holds for 
  $(\mbox{Aff}_3(\F_2), \phi_7,\phi_{8a},\phi_{8b})$.  
  Working more algebraically, as described in \S\ref{calculations},
  one can verify the analogous
  convexity assertion in the presence of $\phi_8$, giving the first sentence of the following
  result.  
    \begin{prop} The tame-wild principle holds for 
  $(\mbox{\rm Aff}_3(\F_2), \phi_7, \phi_8,\phi_{8a},\phi_{8b})$.   In particular,
  to find all $8T48$ extensions with $|\fd_{K_{8a}/F}| \leq B$, one need look
  only at $7T5$ extensions with $|\fd_{K_{7}/F}| \leq B$ and select from among
  the octic resolvents of their $14T34$ quadratic overfields.
  \end{prop}
  \noindent The second sentence comes from an understanding of the algebraic meaning
  of Figure~\ref{pictg1344}.
Associate variables $u$, $a$, and $b$  to $\phi_7$, $\phi_{8a}$, and $\phi_{8b}$
respectively.
  The four sides of the trapezoid $T'_+(\mbox{Aff}_3(\F_2), \phi_7,\phi_{8a},\phi_{8b})$ 
  in the drawn $(u/b,a/b)$ plane correspond to the four faces of the cone 
  $T_+(\mbox{Aff}_3(\F_2), \phi_7,\phi_{8a},\phi_{8b})$ in $(u,a,b)$-space.
  These four faces correspond to the four inequalities on local exponents on the
  left and they translate into divisibility relations among either local or 
  global discriminants on the right:
  \begin{align*}
   u & \leq a  \leq u+b,  &   \fd_{K_7/F} \mid \fd_{K_{8a}/F} \mid \fd_{K_7/F} \fd_{K_{8b}/F}, \\
   u & \leq b  \leq u+a,  &   \fd_{K_7/F} \mid \fd_{K_{8b}/F} \mid \fd_{K_7/F} \fd_{K_{8a}/F}.
  \end{align*}
     For tabulations of all extensions $K_{8a}/F$ with
  $|\fd_{K_{8a}/F}|$ at most some bound $B$, the procedure 
  referred to by the proposition is to look for all $K_{7}/F$ with $|\fd_{K_{7}/F}| \leq B$, 
  take suitable
  square roots to pass from $7T5$ fields to $14T34$ fields, and then 
  use resolvents to obtain the desired $8T48$ fields.   
   
    \subsubsection*{The inertial method for $(S_6,\phi_{6a},\phi_{6b},\phi_{10})$}  
  \label{s6}
  The group $S_6$ has three faithful
  permutation representations of degree at most ten: two sextic ones
  $\rho_{6a}$ and $\rho_{6b}$ interchanged by the outer automorphism of $S_6$, 
  and a decic one $\rho_{10}$ coming from the exceptional
  isomorphism $S_6 \cong PSL_2(\F_9).\gal(\F_9/\F_3) = 10T32 \subset S_{10}$. 
  
  \begin{figure}[htb]
   \[ 
   \begin{array}{c|cccccccccc}   
    & 2A & 2B & 2C & 3A & 3B & 4A & 4B & 5A & 6A & 6B \\
 \hline
P &    222 & 21^4 & 2211 & 33 & 3111 & 42 & 411 & 51 & 6 & 321 \\
   &    21^4 & 222 & 2211 & 3111 & 33 & 42 & 411 & 51 & 321 & 6 \\
 &   2^3 1^4 & 2^3 1^4 & 2^4 1^2 & 3331 & 3331 & 4411 & 442 & 55 & 631 & 631 \\
    \hline
\hat{T} &   6 & 2 & 4 & 6 & 3 & 6 & 4 & 5 & 6 & 5 \\
+ &   2 & 6 & 4 & 3 & 6 & 6 & 4 & 5 & 5 & 6 \\
   &    6 & 6 & 8 & 9 & 9 & 8 & 10 & 10 & 9 & 9 \\
    \hline
 T &   3 & 1 & 2 & 4 & 2 & 4 & 3 & 4 & 5 & 3 \\
\bullet &    1 & 3 & 2 & 2 & 4 & 4 & 3 & 4 & 3 & 5 \\
  &   3 & 3 & 4 & 6 & 6 & 6 & 7 & 8 & 7 & 7 \\
  \multicolumn{11}{c}{\;}
   \end{array}
\]
\includegraphics[width=4in]{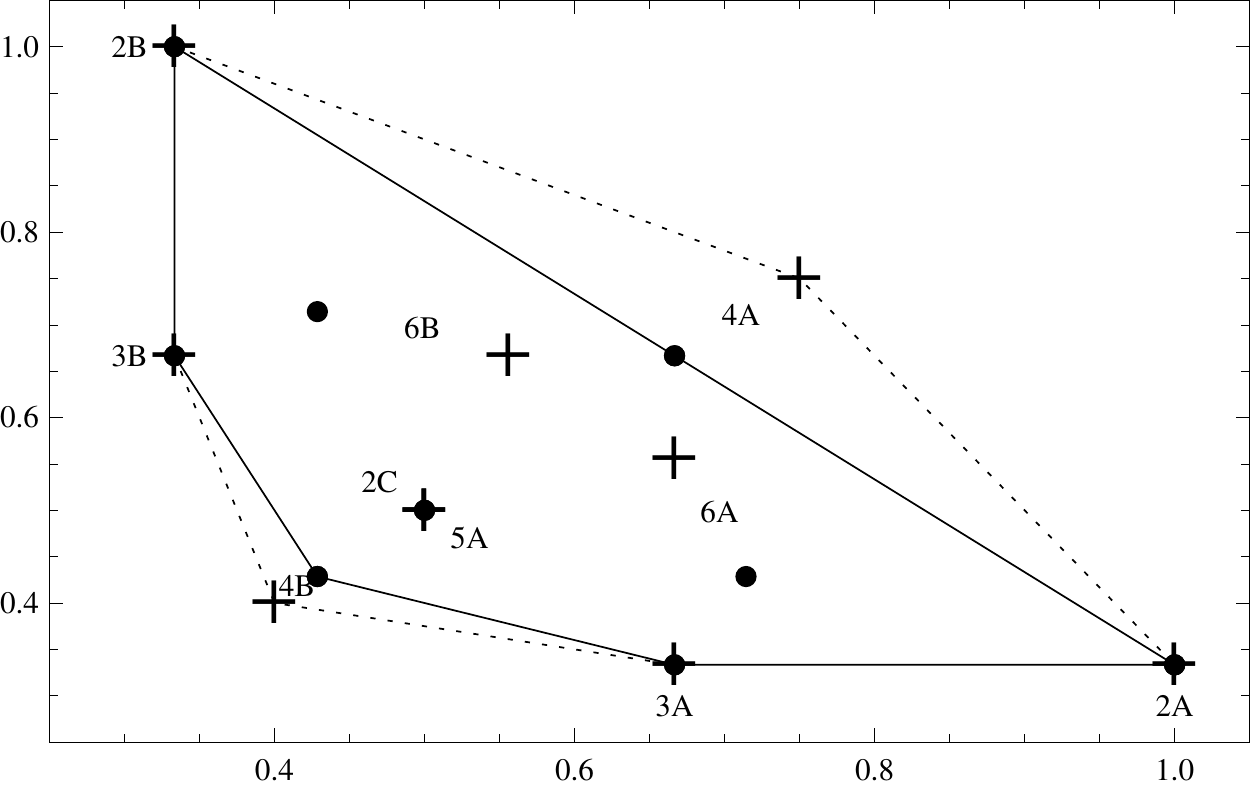}
\caption{\label{picts6}
Top: The partition matrix, broad matrix, and tame matrix
for $(S_6,\phi_{6a},\phi_{6b},\phi_{10})$.   Bottom: the broad hull strictly
containing the tame hull, showing that the broad method does not suffice to prove
the tame-wild principle  for 
$(S_6,\phi_{6a},\phi_{6b},\phi_{10})$.  
}
\end{figure}

  Figure~\ref{picts6} presents our standard analysis of the situation.   Since some $+$'s are outside of the 
  tame hull $T'_+(S_6,\phi_{6a},\phi_{6b},\phi_{10})$, the broad method
  does not suffice for $(S_6,\phi_{6a},\phi_{6b},\phi_{10})$.  
  However after projection to the horizontal axis, the $+$'s 
  are indeed in the convex hull of the $\bullet$'s, so that
  the broad method establishes the tame-wild principle
  for $(S_6,\phi_{6a},\phi_{10})$.   Also the ratios $a/b$ 
  for the $+$ points $(a,b)$ are within the interval $[1/3,3]$
  formed by the ratios for the $\bullet$ points, proving
  the tame-wild principle for $(S_6,\phi_{6a},\phi_{6b})$.
  
  In fact, the tame-wild principle is true for $(S_6,\phi_{6a},\phi_{6b},\phi_{10})$
  as follows.  The only inertial subgroups not covered by previous considerations
  are $I_1=D_4 \times C_2$ and the twin pair  $(I_2,I_3) = (A_4 \times C_2,6T6)$.  
  The orbit partitions in the three cases are $(42,42,442)$, $(42,6,64)$, and $(6,42,64)$.
  The associated conductor vectors are then $(4,4,7)$, $(4,5,8)$, and $(5,4,8)$.
  Their projectivized versions are  $(4/7,4/7)$, $(1/2,5/8)$, and $(5/8,1/2)$.  Since
  these points are visibly in $T'_+(S_6,\phi_{6a},\phi_{6b},\phi_{10})$ 
  the tame-wild principle holds.    We have given this argument to illustrate
  how the inertial method typically applies.  However  in this case 
  the inertial groups $I_2$ and $I_3$ could also have been
  treated by using the techniques from Section~\ref{universalN}, as in 
  fact the universal tame-wild principle holds for $A_4 \times C_2$.

  Summarizing, we have proved the first sentence:
   \begin{prop} The tame-wild principle holds for $(S_6,\phi_{6a},\phi_{6b},\phi_{10})$.
   In particular to find all decic $S_{6}$-extensions with $|\fd_{K_{10}/F}| \leq B$, one
   need only look at sextic $S_6$-extensions with $|\fd_{K_{6a}/F}| \leq B^{2/3}$ 
   and select from among their decic resolvents.
 \end{prop}
  \noindent For the second sentence, note first that the locations of the rightmost and highest points
  of the tame hull in Figure~\ref{picts6} respectively correspond to the equivalent statements
    $\fd_{K_{6a}/F} \mid \fd_{K_{10}/F}$ and $\fd_{K_{6b}/F} \mid \fd_{K_{10}/F}$.  Each of these
  says that to find all decics with absolute discriminant $\leq B$, it suffices to look at all
  sextics up to that bound.   A considerable improvement is to see that the long diagonal boundary
  between them corresponds to $\fd_{K_{6a}/F} \fd_{K_{6b}/F} \mid \fd_{K_{10}/F}^{4/3}$ which implies the
  statement.

\subsubsection*{The broad method for $(W(E_6),\phi_{27},\phi_{36},\phi_{40a},\phi_{40b},\phi_{45})$}   
\label{w6}
As we have seen in \S\ref{left} and 
by the earlier examples of this subsection, 
the broad method
works well in the setting $r=2$.
As $r$ increases, the
difference between $a_\tau$ 
and $\hat{a}_\tau$ becomes more visible, and 
the broad method often fails even when the
tame-wild principle is true, as we just saw for $(S_6,\phi_{6a},\phi_{6b},\phi_{10})$.

A clear illustration of the effectiveness of the 
broad method and its decay with increasing $r$ comes from the Weyl
group $W(E_6)$ of order $51840 = 2^6 3^4 5$ and the 
permutation characters $\phi_{27}$, $\phi_{36}$, $\phi_{40a}$, $\phi_{40b}$,
$\phi_{45}$ corresponding to five maximal subgroups \cite{ATLAS}.  
The broad method immediately shows that the tame-wild principle
for $(W(E_6),\phi_{u},\phi_{v})$ holds for all ten possibilities for 
$\{u,v\}$.  From ten pictures like Figures~\ref{pictg1344} and \ref{picts6}, now quite involved since
$|W(E_6)^{\sharp 0}|=24$, the broad method establishes
the tame-wild principle in exactly four of the ten cases 
$(W(E_6),\phi_u,\phi_v,\phi_w)$ as follows.

\begin{prop} For $\{u,v,w\} = \{27, 36, 40a\}$,
$\{27,40a,40b\}$, $\{36, 40a,40b\}$, and $\{36,40b,45\}$,
the tame-wild principle holds for $(W(E_6),\phi_u,\phi_v,\phi_w)$.
\end{prop}
\noindent Pursuing this situation further with the inertial method would be harder, 
 because $W(E_6)$ has many 2-inertial and 3-inertial 
subgroups.

\subsection{Best counterexamples}
\label{bestcounter}
 
   Let $G$ be a group for which the universal tame-wild principle fails.  
Then there exists a vector $v \in \Q(G^{\sharp})$ for
which the tame-wild principle fails for $(G,\langle v \rangle)$.  There are infinitely
many solutions to  $\phi_1-\phi_2 \in \langle v \rangle$ with 
the $\phi_i$ permutation characters.  So any
failure of the universal tame-wild principle 
can be converted to a failure in the setting $(G,\phi_1,\phi_2)$
of the introduction.   By switching $\phi_1$ and $\phi_2$ if necessary,
it can be converted to a failure of the left tame-wild principle for
$(G,\phi_1,\phi_2)$.  

   However these counterexamples are not guaranteed to have immediate bearing on our applications to tabulating
number fields.  All that is asserted by the failure of the principle for 
$(G,\phi_1,\phi_2)$ is that there exists a pair of local extensions
$(K_1,K_2)/F$ of the given type with
\begin{equation}
\label{divisible}
\fd_{K_2/F}^{\alpha(G,\phi_1,\phi_2)} \mid \fd_{K_1/F} 
\end{equation}
not holding. 
More directly relevant would be  global counterexamples with the $\phi_i$ both coming from
faithful transitive permutation representations and
the extensions $K_i/F$ full in the sense of each having Galois group 
$\gal(K^{\rm gal}/F)$ all of $G$. More demanding still 
is to ask for counterexamples of this sort with $F=\Q$.  
Finally, one can seek examples for which even
the weaker numerical statement
\begin{equation}
\label{lessthan}
|\fd_{K_2/\Q}|^{\alpha(G,\phi_1,\phi_2)} \leq |\fd_{K_1/\Q}|
\end{equation}
fails.  Examples of this explicit nature often do not exist
for a given $G$, and even when they exist they can be 
hard to find.   The rest of this subsection discusses the construction of
global counterexamples built from one of the two local counterexamples
with $I = Q_8$ with slope-content $[2,2,5/2]^2$ from \eqref{q8counters}. 
There are several points of contact with \S\ref{elusive}, but
here we find counterexamples  to \eqref{divisible} over $\Q$.  

\subsubsection*{Inadequacy of $G = \hat{Q}_8$ as a source of global counterexamples} 
The group $\hat{Q}_8$ itself is not a source of global counterexamples of the 
sort we seek because it has only two transitive faithful permutation characters
and the tame-wild principle holds for the corresponding type $(\hat{Q}_8,\phi_8,\phi_{16})$.  To
illustrate the best that can be done with this group, take  
\begin{equation}
\label{intrans}
\begin{array}{r | lrlr}
n & f_n(x) & D_n & G_n & |G_n| \\
\hline
8 &                  x^8+6 x^4-3 & - 2^{16} 3^7 & \hat{Q}_8 & 16 \\
4&                  x^4 +  6 x^2 -3 & -2^6 3^3 & D_4 & 8 \\
2&                  x^2+3   & -3 & C_2 & 2 
\end{array}
\end{equation}
The global and $2$-adic Galois groups of $f_8(x)$ agree, and so 
one has this agreement for the resolvents $f_4(x)$ and $f_2(x)$ as well.
The Galois groups $G_n$ and the field discriminants $D_n$ are as
indicated.   The fields $K_n = \Q[x]/f_n(x)$ belong to transitive characters
$\phi_8$, $\phi_4$, and $\phi_2$ of $\hat{Q}_8$.

  \begin{figure}[htb]
\includegraphics[width=4.5in]{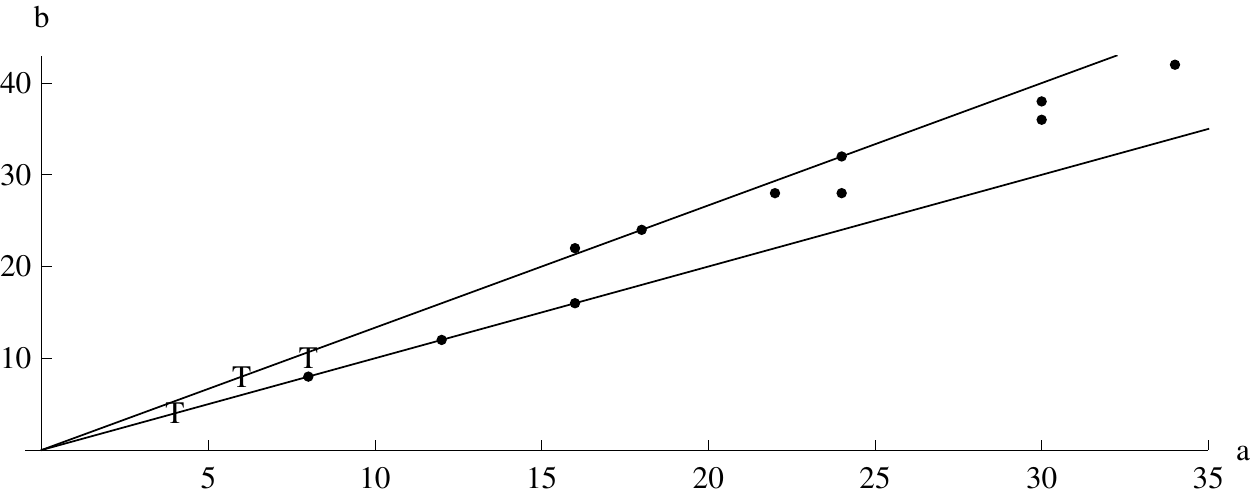}
\caption{\label{8t8pict}  An analog of the introductory Figure~\ref{fivetosix} for the type
$(\hat{Q}_8,\phi_8+\phi_2,\phi_8+\phi_4)$.  The points are exactly
all the possibilities for exponent pairs $(a_2,b_2)$ from wild
$2$-adic ramification over $\Q_2$, and $(16,22)$ is just outside
the tame cone.   }
\end{figure}

Figure~\ref{8t8pict} is an analog of Figure~\ref{fivetosix}, but now for 
$(\hat{Q}_{8},\phi_8+\phi_2,\phi_8+\phi_4)$.    
The algebra pair $(K_8 \times K_2,K_8 \times K_4)/\Q$ yields
the exponent pair $(a_2,b_2) = (16,22)$ which is just outside the
tame cone.   So this pair of algebras indeed 
contradicts \eqref{divisible}, but we are seeking 
counterexamples among pairs of fields.

\subsubsection*{Failure of the inertial method for $(M_{12},\phi_{12a},\phi_{12b})$} 
To get a better global counterexample corresponding to 
the same local counterexample, we need to 
replace $\hat{Q}_8$ by larger groups $G$ containing it.  
An initial key observation is that 
the quaternion group $Q_8$ is the four-point 
stabilizer of the Mathieu group $M_{12} \subset S_{12}$
of order $12 \cdot 11 \cdot 10 \cdot 9 \cdot 8$ in 
its natural action, and also one 
has $Q_8 \subset \hat{Q}_8 \subset M_{12}$.  
On the one hand,  the
given character $\phi_{12a}$ of the Mathieu group
has  decomposition $\phi_8 + \phi_2 + 2$ 
when restricted to $\hat{Q}_8$.   On the other hand, 
there is a twin dodecic character $\phi_{12b}$ coming
from the outer involution of $M_{12}$;  its
restriction to $\hat{Q}_8$ decomposes as $\phi_8+\phi_4$.  

Further group-theoretic facts are necessary
for this situation to give number fields as desired.  
First, the partition matrix and projective
tame matrix of $(M_{12},\phi_{12a},\phi_{12b})$ are as 
follows:
{\renewcommand{\arraycolsep}{2pt}
\[
   \begin{array}{c|ccccccccccccc|c}
\tau &   2A & 2B & 3A & 3B & 4A & 4B & 5A & 6A & 6B & 8A & 8B & 10A & 11AB & Q_8 \\
    \hline
\lambda_\tau(\phi_{12a})   &  2^6 & 2^41^4 & 3^31^3 & 3^4 & 4^22^2 & 4^21^4 & 5^21^2 & 6^2 & 6321 & 84
      & 821^2 & (10)2 & (11)1 & 81^4 \\
\lambda_\tau(\phi_{12b})   &  
 2^6 & 2^41^4 & 3^31^3 & 3^4 & 4^21^4 & 4^22^2  & 5^21^2 & 6^2 & 6321 & 821^2 & 84
      & (10)2 & (11)1 & 84 \\
\hline
c_\tau' &    1 & 1 & 1 & 1 & 8/6 & 6/8 & 1 & 1 & 1 & 10/8 & 8/10 & 1 & 1 & 7/10 \\
   \end{array}
\]
}

\noindent Thus, in the language introduced in \S\ref{tauandI}, 
extensions $(K_{12a},K_{12b})/F$ of full type 
$(M_{12},\phi_{12a},\phi_{12b})$ are 
fraternal twins, this being necessary for
our purposes.  But they are near-identical in 
the sense that the interval $[\alpha,\omega]$
is small, being $[3/4,4/3] = [0.75,1.\overline{33}]$ here, rather 
than the intervals $[1/2,2]$ and $[1/3,3]$ seen in \S\ref{tauandI} for
$\mbox{Aff}_3(\F_2)$-twins and $S_6$-twins respectively.  The 
orbit partitions of $Q_8$ are as indicated
above, yielding $c_{Q_8}(\phi_{12a})/c_{Q_8}(\phi_{12b}) = 7/10 = 0.70$ 
which is outside the interval $[0.75,1.\overline{33}]$.   
Thus the inertial method for proving the tame-wild principle
fails here.

\subsubsection*{Failure of the tame-wild principle for $(M_{12},\phi_{12a},\phi_{12b})$}  
Computing with the slopes $[2,2,2.5]^2$, the $8$'s in the last column above
give discriminant exponent $3 \cdot 2 + 4 \cdot 2.5 = 16$ while the $4$ gives the 
discriminant exponent $3 \cdot 2 = 6$.   So the ratio of wild conductors is 
$16/22 \cong 0.\overline{72}$ which is still outside the interval $[0.75,1.\overline{33}]$. 
Thus the tame-wild principle itself fails  for  $(M_{12},\phi_{12a},\phi_{12b})$.

\subsubsection*{Smaller groups}  We have looked in several places, including the two-parameter family
of \cite{MalMP}, for twin pairs $(K_{12a},K_{12b})$ of $M_{12}$ fields with the
needed quaternionic $2$-adic behavior.   We did not find any, and so we consider
smaller groups as follows as potential sources of counterexamples:
\[
\begin{array}{ccccccccccccc}
             &               & \mbox{Aff}_2(\F_3) && P\Gamma L_2(\F_9) \\
             &               & \rotatebox{270}{$\!\!\! \cong$} && \rotatebox{270}{$\!\!\! \cong$} \\
\hat{Q}_8 & \subset & M_9.2 & \subset & M_{10}.2 \\
\cup && \cup && \cup &&&&&&&&\\
Q_8& \subset & M_9 & \subset & \fbox{$M_{10}$}  & \subset & \fbox{$M_{11}$} & \subset & \fbox{$M_{12}$} &
 \supset & \fbox{$T$}  & \supset & P.  
\end{array}
\]  
The four groups in the middle are boxed to emphasize that they appear in Proposition~\ref{counter} below.

Proceeding from $M_{12}$ to the left,  the groups $M_{11}$ and $M_{10}$ contain $\hat{Q}_8$, 
since $\hat{Q}_8$ has orbit partition $8211$.  Thus, using
  $0.\overline{72} \not \in [0.75,1.\overline{33}]$ exactly as above, 
 the tame-wild principle fails also for $(M_{11},\phi_{11},\phi_{12b})$ and 
 $(M_{10},\phi_{10},\phi_{12b})$.   Here the transitive permutation groups in question 
 are respectively $(11T6,12T272)$ and $(10T31,12T181)$.    The analog of Figure~\ref{8t8pict}
 for $M_{10}$ and $M_{11}$ has the same tame cone, but more dots.  For $M_{12}$ there are many
 more dots, and a symmetry appears with the cone doubling so that its bounding lines have slope
$3/4$ and $4/3$ rather than $1$ and $4/3$.  
 
 Moving further leftward to $M_9$ and $Q_8$ relates our current discussion 
to  our earlier counterexamples.  For $M_9$, the
 transitive groups are $(9T14,12T47)$.  However now $\hat{Q}_8$ is not contained
 in $M_9$ and so we do not have counterexamples over $\Q$.  However the counterexample
 for $(M_9,\phi_{12b},\phi_{72})$ over $\Q(\sqrt{-3})$ from \S\ref{elusive} also gives a counterexample for 
 $(M_9,\phi_9,\phi_{12b})$, as always because the projectivized wild Artin conductor
 $0.\overline{72}$ is not in the tame hull $[0.75,1.\overline{33}]$.   Finally for
 $\hat{Q}_8$ itself we recover \eqref{intrans}, now interpreted as an intransitive counterexample for 
 $(Q_8,\phi_{8},\phi_8+\phi_4)$ over
 $\Q(\sqrt{-3})$.  
 
 The extended groups $M_9.2$ and $M_{10}.2$ corresponding to the pairs $(9T19,12T84)$
  and $(10T35,12T220)$ are natural candidates to support examples over $\Q$ because
  they contain $\hat{Q}_8$.  However they have orbit partitions
   $921$ and $(10)2$ as subgroups of $M_{12}$.    Computation in the column 
   headed by $8B$ then has to be adjusted, with the $2$ in $821^2$ removed.  The
   conductor ratio is then $7/10$ rather than $8/10$ and in fact the inertial method above works to prove 
   the tame-wild principle for $(M_{9}.2,\phi_9,\phi_{12b})$ and $(M_{10}.2,\phi_{10},\phi_{12b})$.
   This phenomenon illustrates the fundamental difficulty in promoting local
   non-transitive counterexamples to global transitive ones with a larger group.
   While wild Artin conductor ratios, here $0.\overline{72}$ stay the same, tame hulls
   increase, here from $[0.75,1]$ for $\hat{Q}_8$ itself to $[7/10,7/6] = [0.70,1.1\overline{6}]$
   for $M_{10}.2$.  
   
   There are other good candidates for global Galois groups.   The  
   $2$-Sylow subgroup $P$ of $M_{12}$ of order $2^6$ 
   is not good for us, because neither
   the given orbit decomposition nor its twin is transitive, both having orbit partition $84$.  
   However an overgroup $T$ of order $2^6 \cdot 3 = 192$ is good, with
   $(\phi_{12a},\phi_{12b})$ remaining a fraternal pair of type $(12T112,12T112)$.
   Our computations have shown that the tame-wild principle fails for 
   $(12T112,\phi_{12a},\phi_{12b})$.   
   
 \subsubsection*{Number fields}   Constructing number fields with nonsolvable Galois groups 
 and prescribed ramification remains a difficult problem despite the increasing attention
  it has been receiving recently.    Just as we have not found $M_{12}$ fields with the
  appropriate quaternionic ramification, we have also not found $M_{11}$ 
  or $M_{10}$ fields.  
  
  In contrast, it is relatively easy to build solvable fields step by step, and we have found
  many explicit pairs $(K_{12a},K_{12b})/\Q$ providing counterexamples
  to the tame-wild principle for  $(12T112,\phi_{12a},\phi_{12b})$.  One such, with 
  tame ramification at the prime number $q=277$, is 
\begin{eqnarray*}
f_{12a}(x) & = & x^{12}+223 x^{10}+14856 x^8+1784 q x^6 +38160 q x^4 +1712 q^2 x^2 +9216 q^2 , \\
f_{12b}(x) & = &  x^{12}+202 x^8+49 q x^4 +4 q^2 .
\end{eqnarray*}
The discriminants are $D_{12a} = 2^{16} 277^8$ and $D_{12b} = 2^{22} {277}^6$, 
with the tame prime $277$ having ramification partitions $\mu_{12a} = 4^22^2$ 
and $\mu_{12b} = 4^21^4$. 

By design, $\fd_{K_{12b}}^{0.75}  \nmid  \fd_{K_{12a}}$.  However the tame ramification at $277$
completely overwhelms the wild ramification at $2$ in terms of magnitudes,  and easily $|\fd_{K_{12b}}|^{.75} \leq  |\fd_{K_{12a}}|$.
 Indeed $|\fd_{K_{12b}}|^{1.15} \approx  |\fd_{K_{12a}}|$.    To improve upon the counterexample
 $(f_{12a},f_{12b})$, one would like examples with $D_{12a} = 2^{16} (p_1 \cdots p_k)^6$
 and $D_{12b} = 2^{22} (p_1 \dots p_k)^8$ so that \eqref{lessthan} is contradicted as well.
 However no such counterexamples exist with $G = 12T112$, as the subgroup $Q_8$ together
 with all elements of type $4^21^4$ generate an index two subgroup of type $12T63$ and
 this subgroup does not contain $\hat{Q}_8$.     Partially summarizing:
\begin{prop}  \label{counter} 
The tame-wild principle for $(G,\phi_{12a},\phi_{12b})$ fails for the groups
$G=M_{12}$, $M_{11}$, $M_{10}$ and $12T112$.  For $G = 12T112$ the pair of 
number fields $(K_{12a},K_{12b})$ contradicts the divisibility statement
\eqref{divisible}, but no pair with $G=12T112$ contradicts the numerical statement \eqref{lessthan}.
\end{prop}
The group-theoretic argument for $12T112$ does not apply to the three $M_n$ and
we expect that there exist pairs $(K_{12a},K_{12b})$ for them contradicting not only
\eqref{divisible} but also  
\eqref{lessthan}.    In general, a closer analysis of the exact range of 
applicability of the tame-wild principle would be interesting.

\bibliographystyle{amsalpha}
\bibliography{jr}

\newcommand{\etalchar}[1]{$^{#1}$}
\providecommand{\bysame}{\leavevmode\hbox to3em{\hrulefill}\thinspace}
\providecommand{\MR}{\relax\ifhmode\unskip\space\fi MR }
\providecommand{\MRhref}[2]{%
  \href{http://www.ams.org/mathscinet-getitem?mr=#1}{#2}
}
\providecommand{\href}[2]{#2}
\begin{thebibliography}{CCN{\etalchar{+}}85}

\bibitem[Ama71]{amano}
Shigeru Amano, \emph{Eisenstein equations of degree {$p$} in a
  {$\mathfrak{p}$}-adic field}, J. Fac. Sci. Univ. Tokyo Sect. IA Math.
  \textbf{18} (1971), 1--21.

\bibitem[CCN{\etalchar{+}}85]{ATLAS}
J.~H. Conway, R.~T. Curtis, S.~P. Norton, R.~A. Parker, and R.~A. Wilson,
  \emph{Atlas of finite groups}, Oxford University Press, Eynsham, 1985.

\bibitem[CGJ{\etalchar{+}}02]{elusive}
Peter~J. Cameron, Michael Giudici, Gareth~A. Jones, William~M. Kantor,
  Mikhail~H. Klin, Dragan Maru{\v{s}}i{\v{c}}, and Lewis~A. Nowitz,
  \emph{Transitive permutation groups without semiregular subgroups}, J. London
  Math. Soc. (2) \textbf{66} (2002), no.~2, 325--333.

\bibitem[CHM98]{conway-hulpke-mckay}
John~H. Conway, Alexander Hulpke, and John McKay, \emph{On transitive
  permutation groups}, LMS J. Comput. Math. \textbf{1} (1998), 1--8
  (electronic).

\bibitem[Fon71]{fontaine}
Jean-Marc Fontaine, \emph{Groupes de ramification et repr\'esentations
  d'{A}rtin}, Ann. Sci. \'Ecole Norm. Sup. (4) \textbf{4} (1971), 337--392.

\bibitem[FP92]{FP}
David Ford and Michael Pohst, \emph{The totally real {$A_5$} extension of
  degree {$6$} with minimum discriminant}, Experiment. Math. \textbf{1} (1992),
  no.~3, 231--235.

\bibitem[FPDH98]{FPDH}
David Ford, Michael Pohst, Mario Daberkow, and Nasser Haddad, \emph{The {$S_5$}
  extensions of degree 6 with minimum discriminant}, Experiment. Math.
  \textbf{7} (1998), no.~2, 121--124.

\bibitem[JLY02]{generic}
Christian~U. Jensen, Arne Ledet, and Noriko Yui, \emph{Generic polynomials},
  Mathematical Sciences Research Institute Publications, vol.~45, Cambridge
  University Press, Cambridge, 2002.

\bibitem[Jon12]{jj-nonic}
John~W. Jones, \emph{Minimal solvable nonic fields}, in preparation, 2012.

\bibitem[JR]{jr-global-database}
John~W. Jones and David~P. Roberts, \emph{A database of number fields}, in
  preparation, website: \url{http://hobbes.la.asu.edu/NFDB}.

\bibitem[JR06]{jr-local-database}
\bysame, \emph{A database of local fields}, J. Symbolic Comput. \textbf{41}
  (2006), no.~1, 80--97, website: \url{http://math.asu.edu/~jj/localfields}.

\bibitem[JR08]{jr-2adicoctics}
\bysame, \emph{Octic 2-adic fields}, J. Number Theory \textbf{128} (2008),
  no.~6, 1410--1429.

\bibitem[JW12]{jj-wallington}
John~W. Jones and Rachel Wallington, \emph{Number fields with solvable {G}alois
  groups and small {G}alois root discriminants}, Math. Comp. \textbf{81}
  (2012), no.~277, 555--567.

\bibitem[Kon39]{kontorovich1}
P.~Kontorovich, \emph{On the representation of a finite group as a direct sum
  of its subgroups i}, Math. Sb. \textbf{5} (1939), 283--286.

\bibitem[Kon40]{kontorovich2}
\bysame, \emph{On the representation of a finite group as a direct sum of its
  subgroups ii}, Math. Sb. \textbf{7} (1940), 27--33.

\bibitem[LMF]{LMFDB}
\emph{L-function and modular forms database}, website:
  \url{http://www.lmfdb.org/}.

\bibitem[Mal00]{MalMP}
Gunter Malle, \emph{Multi-parameter polynomials with given {G}alois group}, J.
  Symbolic Comput. \textbf{30} (2000), no.~6, 717--731.

\bibitem[Mau65]{maus}
E.~Maus, \emph{Existenz $\mathfrak{p}$-adischer {Z}ahlk\"orper zu vorgegebenem
  {V}erzweigungsverhalten}, Ph.D. thesis, Hamburg, 1965.

\bibitem[Ser79]{serre-CL}
Jean-Pierre Serre, \emph{Local fields}, Graduate Texts in Mathematics, vol.~67,
  Springer-Verlag, New York, 1979, Translated from the French by Marvin Jay
  Greenberg.

\bibitem[Suz50]{suzuki1}
Michio Suzuki, \emph{On the finite group with a complete partition}, J. Math.
  Soc. Japan \textbf{2} (1950), 165--185.

\end{thebibliography}

\end{document}